\documentclass[11pt,letterpaper]{amsart}
\usepackage{graphicx}
\usepackage{xy} \xyoption{all}
\usepackage[utf8]{inputenc}
\usepackage{pgf,tikz,pgfplots}
\pgfplotsset{compat=1.14}
\usepackage{mathrsfs}
\usetikzlibrary{arrows}
\usepackage{tikz-network}
\usepackage{todonotes}
\newtheorem{lemma}{Lemma}[section]
\newtheorem{corollary}[lemma]{Corollary}
\newtheorem{theorem}[lemma]{Theorem}
\newtheorem{proposition}[lemma]{Proposition}

\theoremstyle{definition}
\newtheorem{definition}{Definition}[section]

\newtheorem{remarks}[definition]{Remarks}
\newtheorem{remark}[definition]{Remark}

\newtheorem{example}[definition]{Example}

\newcommand{\N}{\mathbb{N}}

\newcommand{\Leo}{\mathbb{L}}
\newcommand{\Le}{\mathbb{L}_K(1, n)}
\newcommand{\Lem}{\mathbb{L}_K(m,n)}
\newcommand{\inlim}{\underrightarrow{\rm lim}}
\newcommand{\fL}{\frak L}

\def\su{\mathop{\sum}\limits_{i=1}^{n}}

\DeclareMathOperator{\End}{End}
\DeclareMathOperator{\Aut}{Aut}
\DeclareMathOperator{\ann}{ann}

\DeclareMathOperator{\Hom}{Hom}  
\DeclareMathOperator{\codim}{codim} 
   
\def\a{\alpha}
\def\b{\beta}
\def\f{\phi}
\def\r{\rho}
\def\s{\sigma}

\def\p{\pi}
\def\l{\lambda}
\def\L{\Lambda}
\def\L^*{\Lambda^*}
\def\g{\gamma} 
\def\d{\delta}
\def\m{\mu}

\def\bB{\bf B}
\begin{document}
\title{Non-commutative factorizations and finite-dimensional representations of free algebras}
\author{P. N. \'Anh}
\address{Hun-Ren R\'enyi Institute of Mathematics, 1364 Budapest, Pf. 127 Hungary} \email{anh@renyi.hu}
\author{F. Mantese}
\address{Department of Computer Science, University of Verona, Strada le Grazie 15,
37134 Italy} \email{francesca.mantese@univr.it}
\thanks{The first author was partially supported by National Research, Development and Innovation Office NKHIH K138828 and  K132951, by both Vietnam Institute for Advanced Study in Mathematics (VIASM) and  Vietnamese Institute of Mathematics. The second author was partially supported by  the National Recovery and Resilience Plan (NRRP), Mission 4 Component 2 Investment 1.1 - Call PRIN 2022 No. 104 of February 2, 2022 of Italian Ministry of University and Research; Project 2022S97PMY (subject area: PE - Physical Sciences and Engineering) "Structures for Quivers, Algebras and Representations (SQUARE)".  She is  moreover member of INDAM - GNSAGA}

\makeatletter
\@namedef{subjclassname@2020}{\textup{2020} Mathematics Subject Classification}
\makeatother
\subjclass[2020]{Primary 16S88, 16S90; secondary 16G20, 16P50.}
\keywords{free algebras, irreducibile polynomials, atomic factorizations, Leavitt algebras, finitely presented modules, flat rings of quotients}
\date{\today}
\maketitle
\begin{abstract}
A very first step to develop non-commutative algebraic geometry is the arithmetic of polynomials in non-commuting variables over a commutative field, that is, the study of {elements in} free associative algebras. This investigation is presented as a natural extension of the classical theory in one variable by using Leavitt algebras, which are localizations of free algebras with respect to the Gabriel topology defined by an ideal of codimension 1. In particular to any polynomial in $n$ non-commuting variables with non-zero constant term we associate a finite-dimensional module over the free algebra of rank $n$, which turns out to be simple if and only if the polynomial is irreducible. This approach leads to new insights in the study of the factorization of polynomials into irreducible ones and other related topics, such as an algorithm to divide polynomials or to compute the greatest common divisor between them, or a description of similar polynomials.
The case of polynomials with zero constant term reduces to an open question whether the intersection of {nonunital} subalgebras of codimension 1 in a free associative algebra is trivial, that is, 0. 
\end{abstract}
\section{Introduction}
\label{int} 
The classical theory of polynomials in one variable over a commutative field $K$  assigns to every  irreducible polynomial its residue field, i.e., its finite-dimensional field extension of remainders by Euclidean division, and shows that every polynomial is uniquely a product of irreducible ones up to permutation and units. The 
division algorithm to compute the greatest common divisor of polynomials leads to a complete picture of the lattice of the ideals of $K[x]$. 
Moreover, the irreducible representations of  $K[x]$ are explicitly constructed from the irreducible polynomials in $K[x]$. 

Analogous results for polynomials in several commutative variables are no longer true. For instance, the factor algebra of an irreducible polynomial  in  $K[x_1, \cdots, x_n]$ for $n\geq 2$ is not a field extension of $K$ and is never of finite dimension.

For $n\geq 2$  polynomials in non-commuting variables are elements of the free associative  algebra $K\langle x_1, \cdots, x_n\rangle$. Factorization problems in  $K\langle x_1, \cdots, x_n\rangle$ have been deeply studied by  Cohn in \cite{c2}, \cite{c4}, \cite{c5}, \cite{c6}, where he showed that  free algebras have a weak division algorithm,  they are weak B\'ezout free ideal rings, and they are  unique factorization rings, i.e any element is {the} product of irreducible ones in a unique way up to units and similarity. Roughly speaking, weaker versions of the relevant properties of $K[x]$ to be a Euclidean domain, a principal ideal domain and a unique factorization domain hold for the free algebra $K\langle x_1, \cdots, x_n\rangle$. Moreover,  by considering the category  of strictly cyclic modules or more generally  of (specific) torsion modules over {free ideal rings}, Cohn  succeeded to show a categorical simplicity of factor modules of the free algebra  by irreducible polynomials. However, Cohn's approach {leaves} open the problem to  parametrize the irreducible representations of $K\langle x_1, \cdots, x_n\rangle$ by irreducible polynomials, as in the classical case of one variable.

The aim of our work is to present a theory of polynomials 
which coincides with the classical one in case of one variable. {The starting point is the Cohn's unique normal form \cite{c3}
$$\l=k+\sum\limits_{i=1}^n x_i\l_{x_i} \quad (k\in K; \, \l_{x_i} \in K\langle x_1, \cdots, x_n\rangle)$$
of any polynomial $\l\in K\langle x_1, \cdots, x_n\rangle$, where $\l_{x_i}=\l_i$ is called  the \emph{cofactor} with respect to the $i$-th axis (i.e., variable) $x_i$. Hence taking cofactors behaves like taking generalized partial differentiation or truncated partial inverse along axes. One can iterate this process to obtain further cofactors of $\l$ until getting nonzero constants.}
 For any polynomial $\l$ in $K\langle x_1, \cdots, x_n\rangle$  with {nonzero} constant term we define an easily handled finite-dimensional vector space $V_\l$, {spanned by all nonzero {iterated} cofactors of $\l$}, which admits a natural left module structure over the free algebra. This new invariant becomes a good tool in the study of the polynomial $\l$ itself, independent on $K\langle x_1, \cdots, x_n\rangle$. For instance, the module $V_\l$ helps to decide the irreducibility of $\l$, to compute its similar polynomials and its irreducible {factors} or indecomposable {factors}, i.e., primary decompositions,  and to compute other invariants such as endomorphism rings, submodule lattices, etc. In particular, $V_\l$ is simple if and only $\l$ is irreducible. {More generally, if $\l=\p_1\cdots \p_l$ is a factorization of $\l$ as a product of irreducible polynomials $\p_1, \cdots, \p_l$, then $l$ is the length of $V_\l$ and $\{V_{\p_1}, \cdots, V_{\p_l}\}$ is the set of composition factors of $V_\l$. If $\g$ is another polynomial with nonzero constant term, then $V_\l$ and $V_\g$ are isomorphic $K\langle x_1, \cdots, x_n\rangle$-modules if and only if the $K\langle x_1, \cdots, x_n\rangle$-modules $K\langle x_1, \cdots, x_n\rangle/K\langle x_1, \cdots, x_n\rangle\l$ and $K\langle x_1, \cdots, x_n\rangle/K\langle x_1, \cdots, x_n\rangle \g$ are such.} Therefore $V_\l$ encodes the structure of $\l$. 
 
The module structure of  $V_\l$ is inspired by ideas coming from Leavitt  algebras, {which are universal with
respect to an isomorphism property between finite rank free modules \cite{leav1}.} This is not surprising, since free algebras and  Leavitt algebras  are tightly  connected by the concept of localization (see Section \ref{leat}). This means that even if these two algebraic structures are deeply different in view of their ring properties (for instance  Leavitt algebras are far {from being} domains), their module categories are closely related.
 Moreover, any free algebra can be seen naturally in two  ways as a subalgebra of a  Leavitt algebra. This fact reflects the decisive role of Leavitt algebras in the investigation of both right and left factorizations of polynomials. By symmetry it is enough to study right factorizations. Since  Leavitt algebras are {one-sided} B\'ezout rings, i.e. any finitely generated {right (or left)} ideal is principal, moving to  this larger setting it is possible to define the  \emph{generalized greatest common divisor}, shortly, the \emph{greatest common divisor}, for finitely 
 many polynomials $\l_i $ in $ K\langle x_1, \cdots, x_n\rangle$: it is  the  uniquely determined, up to scalars, polynomial  with {nonzero} constant term which is a generator of the left ideal of the extended Leavitt algebra generated by the $\l_i$'s.

 In this manner
we succeed to have an algorithm computing the greatest common divisor of polynomials and to divide polynomials in  $K\langle x_1, \cdots, x_n\rangle$ analogous to the Euclidean algorithm for $K[x]$. This will be a crucial tool for our theory. {Another basic notion in our treatment is that of similarity, which  is naturally inspired by the similarity of linear transformations. Namely, two elements are \emph{similar} if the corresponding factors by the left (right) ideals generated by them are isomorphic (see  \cite{c6}). We will show that two elements $\l$ and $\g$ with nonzero constant term are similar if and only if their  associated finite-dimensional modules $V_\l$ and $V_\g$ are isomorphic.}
 \vskip 0.25cm  
The article is organized as follows.
\vskip 0.25cm 
In Section \ref{prif} we state the main result, Theorem \ref{goal1}, that every irreducible polynomial with constant term in $n$ non-commutative variables defines a simple finite-dimensional module over  the free algebra of rank $n$, whence a finite-dimensional division algebra, opening several exciting related questions.  Consequently, every polynomial with nonzero constant term factors uniquely into products of irreducible up to units and similarity, in the sense that if  $\l=\d_1\cdots \d_s=\p_1\cdots \p_t$ where the $\d_i$'s and the  $\p_j$'s are irreducible, then $s=t$ and there exists a permutation $\s$ such that $\d_i$ is similar to $\p_{\s(i)}$, even if $\l$ is no longer equal to a permutated product $\prod\limits_{i=1}^s\p_{\s(i)}$. 
Since irreducibilty does not change under automorphisms, polynomials with zero constant term may become ones with nonzero constant after suitable substitutions. The feasibility of this trick is an open question whether {or not} 
the intersection of {nonunital} subalgebras of codimension 1 is 0. 

Section \ref{leat} develops shortly the theory of Leavitt algebras necessary for Theorem \ref{goal1} and presents the interaction between Leavitt algebras and free algebras in terms of  their module categories. The importance of Leavitt algebras is underlined by the fact that they contain the pair of free algebras $K\langle x_1, \cdots, x_n\rangle$ and $K\langle x^*_1, \cdots, x^*_n\rangle$ together with both the isomorphism and anti-isomorphism induced canonically by the assignment $x_i\mapsto x^*_i\, (i=1, \cdots, n)$.
This parallel reflects precisely both the left and the right factorizations of polynomials. Hence we deal mainly with the right factorizations of polynomials.

Section \ref{pro} is  devoted to prove Theorem \ref{goal1} with additional results and  consequences on the arithmetic of polynomials in non-commutative variables.

{Section \ref{fire} presents applications of Theorem \ref{goal1} to companion structures of polynomials with nonzero constant term, like their similarity classes, atomic factorizations etc. Some important problems concerning irreducible polynomials and finite-dimensional non-commutative division algebras are stated.}

{Section \ref{infirank} contains further problems, discusses primary decompositions of a polynomial as indecomposable decompositions of its module of remainders and shows that Theorem \ref{goal1} remains true when the investigated polynomial is considered as an element of a free algebra of an arbitrary rank. Hence the study of an individual polynomial via the modules of remainders is independent of the free algebra which contains it.}

\section{Atomic factorizations in free algebras}
\label{prif}

Throughout this work $K$ denotes an arbitrary commutative field, rings are unital $K$-algebras, and modules are unitary whence $K$-spaces, too. {When  we discuss to \emph{nonunital subalgebras}  of a given  $K$-algebra $A$,  we mean multiplicatively closed subspaces  without identity.} A ring $A$ is \emph{projective-free} if projective modules are free. {It is worth to note that being projective-free is left-right symmetric.} Two elements $a, b\in A$ are \emph{similar} (over $A$) if the left $A$-modules $A/Aa, A/Ab$ are isomorphic. The \emph{codimension} of a subspace is the dimension of its direct complement. The \emph{colength} of a submodule is the length of its factor module.

 A non-zero element of a ring $A$ is called
\emph{irreducible}, or {\emph{an atom} in Cohn's terminology \cite{c2}}, if it cannot be {written as
a product of two nonzero non-invertible elements}. This definition is clearly preserved by automorphisms. A factorization of an element as a product of irreducible ones is called an \emph{atomic factorization}. {Notice that in the non commutative setting irreducible and prime elements may be different. For example, the equality  $\l=(x_1x_2x_3+x_1+x_3)(1+x_2x_1)=(1+x_1x_2)(x_3x_2x_1+x_1+x_3)$ analyzed  in \cite{v1} shows that the irreducible element  $1+x_1x_2$ divides $\l$ but it divides neither $1+x_2x_1$ nor $x_1x_2x_3+x_1+x_3$. Indeed the non-commutativity requires  to take into account both the left and  the right division in the definition of primeness.  Moreover, although prime elements are trivially irreducible, their existence is not necessarily guaranteed. Conversely, even if irreducible elements are not necessary prime, their existence is ensured for instance by the existence of a degree function, as in the case of polynomials in non-commuting variables, where any non-invertible element admits at least one atomic factorization. Thus in general it makes more sense to look for atomic than prime factorizations.}
\vskip 0.25cm
A free $K$-algebra $\Lambda=\Lambda_n$ of rank $n$ determines uniquely $K$ via
$K=U(\Lambda)\cup \{0\}$ and $n={\rm gldim}\, (\Lambda/[\Lambda, \Lambda])$, where $U(A)$ is the unit group of the ring $A$ and $[A, A]$ is the commutator ideal. The equality $n={\rm gldim}\, (\Lambda/[\Lambda, \Lambda])$ is the famous syzygy theorem of Hilbert computing the global dimension of a commutative polynomial algebra in $n$ variables. However, ${\rm gldim}\, \Lambda_1=1$ holds obviously. 
If a free generating set $\{x_1, \cdots, x_n\}$ is designated, i.e $\Lambda=K\langle x_1, \cdots, x_n\rangle$, one gets a normal form for the elements in $\Lambda$. In other words, a free generating set gives a polynomial description for the elements of $\Lambda$. {The two-sided ideal} 
 $I=\sum\limits_{i=1}^n x_i\Lambda=\sum\limits_{i=1}^n \Lambda x_i$ is a nonunital free algebra on $\{x_1, \cdots, x_n\}$ which is also a free $\Lambda$-module of rank $n$ on each side. However, neither the set $\{x_1,\cdots, x_n\}$ nor $I$ can be recovered from $\Lambda$. In fact, any free generating set describes an automorphism of $\Lambda$ up to a permutation. {As we will see in the sequel,  for the purpose of our work the question whether there is an element of 
$\Lambda$ whose normal form with respect to any free generating set  is always a polynomial with  zero constant term is crucial. Notice that this  is equivalent to the question whether the intersection of  two-sided ideals  of codimension 1 in $\Lambda$ is trivial, that is, 0. Maximal two-sided ideals of codimension 1 in $\Lambda$ are exactly nonunital subalgebras of codimension 1, as it is easy to see.  Hence it is an open, interesting question whether the intersection of nonunital subalgebras of codimension 1 in $\Lambda$ is 0}.

In the sequel,   $\Lambda^*=\Lambda^*_n$ will denote the free algebra on the generating set  $\{x^*_i, \cdots, x^*_n\}$, so that $\Lambda=K\langle x_1, \cdots, x_n\rangle$ and  $\Lambda^*=K\langle x^*_1, \cdots, x^*_n\rangle$ with $I^*=\sum\limits_{i=1}^n \Lambda^*x^*_i= \sum\limits_{i=1}^n x^*_i\Lambda^*$. {Both the isomorphism and anti-isomorphism from $\Lambda$ onto $\Lambda^*$ induced by $x_i\mapsto x^*_i \,\,(i=1, \cdots, n)$ are called \emph{canonical}, making the presentation lucidly. Therefore an arbitrary monomial $x_{i_1}\cdots x_{i_l}$ goes to $x^*_{i_1}\cdots x^*_{i_l}$  $(x^*_{i_l}\cdots x^*_{i_1})$ under this canonical (anti)isomorphism.}

The free monoid $\bB$ on $\{x_1, \cdots, x_n\}$ provides the most natural $K$-basis of $\Lambda$. For $b=x_{i_1}\cdots x_{i_m}\in \bB$ of \emph{length} or \emph{degree}  $|b|=m$, monomials $h_b(l)=x_{i_1}\cdots x_{i_l}$ and $t_b(l)=x_{i_{l+1}}\cdots x_{i_m}$ are called the \emph{head} of length $l$ or the \emph{tail} of \emph{colength} (\emph{codegree}) $l$ of $b$, respectively. We {define}  
  $h_b(0)=1=t_b(|b|)$ and $h_b(|b|)=b=t_b(0)$. Thus $1\in \Lambda$ is an empty monomial of length 0. The \emph{transpose} 
${^\tau}b$ of $b$ is $x_{i_m}\cdots x_{i_1}$. 
The \emph{(total) degree} $\deg \l$ of $\l \in \Lambda$, denoted also by  $|\l|$, is the greatest length of its monomials. Hence $|0|=-\infty$ and $|k|=0$ for $0\neq k\in K$.   One can write each
 $\l\in \Lambda$ uniquely in the form
\begin{align}\label{eq1}
\l=k+\sum\limits_{i=1}^n x_i\l_{x_i} \quad (k\in K; \, \l_{x_i} \in \Lambda).
\end{align}
The polynomial $\sum x_i\l_{x_i}=\underline{\l}$ is called the \emph{main part}  of $\lambda$
while, according to Cohn \cite{c3}, $\l_{x_i}$ is the \emph{right cofactor}  
of \emph{colength} 1 with respect to $x_i\, (i=1, \cdots, n)$. {Moreover, a monomial $b$ in the normal form \eqref{eq1} of $\l$ is called a \emph{maximal monomial of} $\l$ if it is not a proper head of another monomial of $\l$ in \eqref{eq1}.} The polynomial $\l$ is \emph{with} or \emph{without constant} according to $k\neq 0$ or $k=0$, respectively. An element $\lambda$ with $k=1$ is also called a \emph{comonic polynomial}.

Right cofactors $(\l_{x_i})_{x_j}$ of $\l_{x_i}$, denoted as $\l_{x_ix_j}$,   are called \emph{cofactors of colength} 2 of $\l$, and so on until getting nonzero constants. In other words, if $b=x_{i_1}\cdots x_{i_l}$ is a monomial and if $b_1, \cdots, b_m$ are the monomials of $\l$ having the head $b$ with coefficient $k_1, \cdots, k_m$, respectively, then the right cofactor $\l_b$ of colength $l=|b|$ defined by $b$ is  
\begin{align}\label{eq2}
\l_b=\left (\l_{x_{i_1}}\right )_{x_{i_2}\cdots x_{i_l}}=\cdots=\sum\limits_{i=1}^m k_it_{b_i}(|b|).
\end{align}

Cofactors $\l_b$ are defined only for heads $b$ of maximal monomials in $\l$ and their constant term may be 0. Hence for convenience we make the convention 
that cofactors of constants are
0 and put $\l_b=0$ if $b$ is not a head of any maximal monomials of $\l$. Consequently, forming the cofactor with respect to $x_i$ provides a {$\star$-action of $\Lambda$ on itself, that is, $x_i\star \l=\l_i$ for every $\l \in \Lambda$}, which behaves like a generalized differentiation on the axis $x_i$ or a {truncated partial inverse} to the direction $x_i$. In fact, this action makes $\Lambda$ a left $\Lambda$-module which is an essential extension of its unique minimal trivial $\Lambda$-submodule $K$.
Hence we have
\begin{proposition}\label{leavittes0} Forming right cofactors with respect to variables
$x_i \, (i=1, \cdots, n)$ makes $\Lambda$ a left module over itself, denoted as $\Lambda_\partial$, 
which is an essential extension of its unique minimal trivial submodule $K$.
\end{proposition}

\begin{proof} {Let $\l \in \Lambda$ be an arbitrary proper polynomial, i.e., $|\l|\geq 1$. Then for a maximal monomial $b$ of $\l$ which has a nonzero coefficient $k_b\in K$, one has $^{\tau}b\star \l=\l_b=k_b\neq 0$. This shows that $K$ is the smallest essential submodule of $\Lambda_\partial$, completing the proof.}
\end{proof}

Notice that Proposition \ref{leavittes0} could be a starting point in the investigation of injective envelopes of finite-dimensional modules over free algebras.

{For example,} Proposition \ref{leavittes0} shows that when $\Lambda=\Lambda_1=K[x]$, forming cofactors defines a module structure on $K[x]$, namely, the module  $K[x]_\partial=\Lambda_\partial$ which is the injective hull of the simple trivial module $K$ independently on the characteristic of $K$. By this reason $\Lambda_\partial$ can be called  \emph{generalized directional Leavitt module}.

{It turns out that given any polynomial $\l \in \Lambda$ with nonzero constant term, one can modify the $\star$-action defined above to obtain a powerful tool for the study of $\l$ itself. Namely,} any polynomial $\l=k_\l+\underline{\l}$ such that $ 0\neq k_\l\in K$, makes $\Lambda$ a left module over itself, denoted by $\Lambda_\l$, via linear transformations induced by putting
\begin{align}\label{de1}
x_i\star_\l 1=-k^{-1}_\l \l_{x_i} \quad \&  \quad
  x_i\star_\l b=\begin{cases}  b_{x_i}=t_b(1) & \text{if}\qquad x_i=h_b(1),\\ 0& \text{if}\qquad x_i\neq h_b(1)\end{cases}
\end{align}
 for  $i=1, \cdots, n$ and $1\neq b\in \bB$. Thus $x_i\star_\l 1=0$ holds if $x_i$ is not a head of any monomial of $\l$, or equivalently, $\l_{x_i}$ is not defined, that is, $\l_{x_i}=0$ by our convention. Since $x_i\star_\l \l=x_i\star_\l k_\l+\l_i=0$ holds for all indices $i$, $K\l$ is isomorphic to the trivial $\Lambda$-module. Furthermore, for any $\g=k_\g+\underline{\g} \in \Lambda$ one has $x_i\star_\l \g \l=x_i\star_\g (k_\g\l+\underline{\g}\l)=\g_i\l$ whence $\Lambda \l$ is a submodule of $\Lambda_\l$.
In fact, $\Lambda\l$ contains the smallest trivial submodule $K\l$ and it is isomorphic to the generalized directional Leavitt module $\Lambda_\partial$ via the map $\g\l \mapsto \g$.

Notice that if   $\l$ is comonic then for any $\g=k+\sum x_i\g_{x_i}\in \Lambda$,  where $k\in K$,  we have  $x_j\star \g=-k\l_{x_j}+\g_{x_j}$. 

\medskip

{We are now in position to define one of the most important objects in our investigation.}
\begin{definition}\label{de2}
Let $\l \in\Lambda$ be a polynomial with nonzero constant term. Then $V_{\l}$ is  the $K$-subspace of $\Lambda$ generated by all cofactors of $\l$.
\end{definition}

{ In particular,   $ 0\neq \l \in K$ implies  $V_\l=0$ as well as  $|\l|=1$ implies  $V_\l=K$. If $|\l|\geq 1$,  then  $V_\l$ has finite dimension. {For example, if $\l=1+x^2_1+x_1x_2+x^2_2x_3+x_3x_2$, then $1, x_1+x_2, x_2x_3, x_2, x_3$ are cofactors of $\l$ and so $V_\l$ has dimension 5 with a basis $\{1, x_1, x_2, x_2x_3, x_3\}$. For $\l=1+x_1+x_2+x^2_1+x_1x_2+x_2x_1+x^2_2$, then $\{1, 1+x_1+x_2\}$ is the set of all cofactors of $\l$ and $V_\l$ has dimension 2 with a basis $\{1, x_1+x_2\}$. It is a challenging and exciting problem to determine or to estimate the dimension of $V_\l$. }}

{\begin{proposition}
Let  $\l \in\Lambda$ be a polynomial with nonzero constant term. Then $V_{\l}$ is a left $\Lambda$-submodule of $\Lambda_\l$.
\end{proposition}
\begin{proof}
It follows directly  from the definition of the   $\star_\l$-action, since if $\l_b$ is a cofactor of $\l$, then $x_i\star_\l \l_b$ is  a linear combination of cofactors of $\l$.
\end{proof}}

{For instance, consider $\l=1+x^2_1+x_1x_2+x^2_2x_3+x_3x_2$ and  $V_\l=\langle 1, x_1+x_2, x_2x_3, x_2, x_3\rangle$, as before. Then $x_1\star_\l 1=-x_1-x_2$, $x_1\star_\l (x_1+x_2)=1$,  $x_1\star_\l x_2x_3=0$, $x_1\star_\l x_2=0$, $x_1\star_\l x_3=0$. Similar computations for $x_2$ and $x_3$ instead of $x_1$.}

Finally notice that  $V_\l$ and  $\Lambda\l$ are both submodules of $\Lambda_\l$ with trivial intersection, and it is straightforward to  see that if $\g\in V_\l$ then $V_\g\subseteq V_\l$. 
 
\begin{remark}\label{leftright} One can define dually \emph{left cofactors} of $\l \in \Lambda$, and so  left cofactors of higher colength,   by using the normal form
\begin{align}\label{eq20}
\l=k+\sum\limits_{i=1}^n {_{x_i}\l}x_i \quad (k\in K; \, _{x_i}\l \in \Lambda).
\end{align}
Consequently, one obtains that $\Lambda$ admits a right $\Lambda$-module structure with respect to the right $\star$-action by forming left cofactors, that is, for every
$b \in \bB $ and $x_i\, (i\in \{1, 2, \cdots, n\})$ one has
\begin{align}\label{eq21}
  b\star x_i=\begin{cases} h_b(|b|-1) & \text{if}\qquad x_i=t_b(|b|-1),\\ 0& \text{if}\qquad x_i\neq t_b(|b|-1)\end{cases}.
\end{align} 
In particular, $x_i\star k=0$ for every scalar $k\in K$. 
We denote this module structure of $\Lambda$ as $_\partial\Lambda$.

Furthermore, any polynomial  $\l=k_\l+\sum\limits_{i=1}^n {_{x_i}}\l x_i, \, 0\neq k_\l\in K$, makes $\Lambda$ a right $\Lambda$-module, denoted by $_\l\Lambda$, by putting
\begin{align}\label{de1}
1 \star_\l x_1=- {_{x_i}\l k^{-1}_\l}  \,\,\,\,  \&  \,\,\,\,
  b  \star_\l  x_1 =\begin{cases}  h_b(|b|-1) & \text{if}\quad x_i=t_b(|b|-1),\\ 0& \text{if}\quad x_i\neq t_b(|b|-1)\end{cases}
\end{align}
 for  $i=1, \cdots, n$ and $1\neq b\in \bB$. The vector subspace $W_\l$  generated by the right cofactors of $\l$ becomes
 a finite-dimensional right $\Lambda$-submodule of $_\l\Lambda$. We will need both $V_\l$ and $W_\l$ to compute atomic factorizations of $\l$ in Subsection \ref{fac}.
\end{remark}

With the notation above, the main goal of this work is 
\begin{theorem}\label{goal1} Let $\g, \l \in \Lambda$ be {comonic} polynomials of  positive degree and  $V_\g$ and $V_\l$ the finite-dimensional left $\Lambda$-modules   with respect to the $\star$-action defined by $\g$ and  $\l$, respectively. Then:
\begin{enumerate}
\item $\l$ is an irreducible polynomial if and only if $V_\l$ is a simple $\Lambda$-module. 
\item If $\l=\p_1\cdots\p_m$ is a factorization of $\l$ into a product of irreducible polynomials, then $m$ is the length of $V_\l$ and its composition factors are isomorphic to the simple modules $V_{\p_i}$, $i=1, \cdots m$. In particular $m$ is an invariant of $\l$.
\item $V_\l\cong V_\g$ if and only if  $\Lambda/\Lambda\g \cong \Lambda/\Lambda\l$, that is, $\g$ and $\l$ are similar over $\Lambda$.
\end{enumerate} 
\end{theorem} 
$\End(V_\p)$ is called the \emph{(division) finite-dimensional algebra} associated to a  {comonic} (irreducible) polynomial $\p\in \Lambda$. Theorem \ref{goal1} suggests a question of central
importance, { namely,} to characterize comonic irreducible polynomials $\l$ such that $\End_\Lambda( V_\l)$ is non-commutative. We shall discuss  polynomials with zero constant term in Sections \ref{pro} and \ref{fire}.

{Theorem \ref{goal1} suggests how to extend the concept "divisibility" and hence "primeness" in the non-commutative setting. A nonzero, non-invertible element $\d$ \emph{divides} another element $\l$ if $\l$ admits a factorization $\l=\a \b \g$ such that $\d$ is one of the $\a$, $\b$ or $\g$. A nonzero, non-invertible element $\p$ is called \emph{prime} if $\p$ has the property: $\p$ divides a product $\a \b$ if and only if there is a polynomial $\d$ such that $\d$ divides
either $\a$ or $\b$ and $\d$ is in addition similar to $\p$. Then Cohn's Theorem III.3.2 \cite[1st ed.]{c2} shows for polynomials, even more generally, for elements in atomic 2-firs, that $\p$ is prime if and only if it is irreducible.} 
\begin{remark}\label{nre1} As we will see in Section~\ref{pro}, the definition of the $\star$-action in $\Lambda=\Lambda_n$ is inspired by the multiplication by the elements  $x_i^*\, (i=1, \cdots, n)$ in the associated Leavitt algebra $L(1, n)$. In case of $n=1$, so that $\Lambda_1=K[x]$, the extended Leavitt algebra is the algebra of Laurent polynomials $K[x, x^{-1}]$ and $x^*$ can be seen as $x^{-1}$. Assume $\l=1+k_1x+\cdots+k_lx^l$ is a comonic  polynomial in $K[x]$, so that  $\l=1+x\l_x$ and $1=-x\l_x+\l$.  If $g=k+xg_x$ is any polynomial in $K[x]$, then $x\star_\l g=-k\l_x+g_x= x^{-1}\cdot g-kx^{-1}\l$, and  hence the $\star_\l$-action  coincides with the multiplication by $x^{-1}$ modulo $K[x, x^{-1}]\l$. Consequently, using the $\star_\l$-action on $K[x]$,  one can see that $g$ is a multiple of $\l$, i.e. $g\in K[x]\l$, if and only if $x^{|g|}\star_\l g=0$.
Moreover, denoting by  
$$\g=x^l+k_1x^{l-1}+\cdots +k_l=x^l\l(x^{-1})$$ 
the reciprocal polynomial of $\l$, we get $K[x, x^{-1}]\l=K[x, x^{-1}]\g(x^{-1})$ establishing a one-to-one correspondence between remainders modulo $\l$ in $K[x]$ and $\g(x^{-1})$ in $K[x^{-1}]$, respectively. 

$V_\l$ is clearly the $K$-subspace of $K[x]$ spanned by $\{1, x, \cdots, x^{l-1}\}$  and it is clearly closed under the $\star_\l$-action induced by the rule $x\star_\l 1=-\l_x$ (so that, for instance, $x\star_\l x=1$). In particular, $V_\l$ is isomorphic to the  $K[x^{-1}]$-module $K[x^{-1}]/K[x^{-1}]\g(x^{-1})$ with respect to the $K$-algebra isomorphism $K[x]\rightarrow K[x^{-1}]:x\mapsto x^{-1}$.
So $V_\l=\Lambda\star_\l 1$  
is a cyclic module and the annihilator of $1\in V_\l$ is the ideal $K[x]\g$. Indeed,  from $x\star_\l x=1$ it follows that $k_ix^{l-i+1}\star_\l 1=x^{l-1}\star_\l k_ix^{i-1}$ and 
$$\g \star_\l 1=(x^l+k_1x^{l-1}+\cdots +k_l)\star_\l 1=x^l\star_\l 1+x^{l-1}\star_\l \l_x=$$
$$=x^{l-1}\star_\l (-\l_x+\l_x)=x^{l-1}\star_\l 0=0.$$
Hence in general $V_\l$ is not isomorphic to $K[x]/K[x]\l$ even when $\l$ is an irreducible polynomial.

In case $\l$ is irreducible, the ideal $K[x]\g$ is maximal, and hence $V_\l$ is isomorphic to the irreducible representation $K[x]/K[x]\g$.
Therefore, in case of one variable, Theorem~\ref{goal1} provides a parameterization of the irreducible representations of $K[x]$ alternative to the classical one. Namely, one uses the reciprocical $\g$ of an irreducible polynomial $\l$ for this job.

To conclude Remark \ref{nre1} notice that $V_\l$ can be identified  with the space of remainders by the Euclidean division by $\l$. In the classical approach using the factor algebra $K[x]/K[x]\l$, in order to compute the remainder of $g$ by $\l$  without carrying over the division, one has to substitute $x^l$ by $(-k_0-k_1x-\cdots-k_{l-1}x^{l-1})$. Unfortunately this approach is not extendable to the case of several non-commutative variables. However, using the module  $V_\l$ and the $\star_\l$ action, in order to find the remainder  of  $g$ by $\l$, we can compute recursively   $x\star_\l g$ until we get an element in $V_\l$. As we will see in {Section \ref{pro} where the proof of Theorem \ref{goal1} is presented,} this approach can be generalized for several non-commutative variables {and is, in fact, the crucial idea of the proof.}
\end{remark}

From now on, let us  assume  $n\geq 2$. Then Theorem \ref{goal1} shows a relatively transparent and simple way assigning to every proper and comonic  polynomial $\l$  a finite-dimensional $\Lambda$-module $V_\l$ whose endomorphism ring is a finite-dimensional algebra. $V_\l$ plays  the role of "remainders" in a generalized Euclidean division algorithm by the comonic polynomial $\l$. This approach provides also an easy algorithm to compute the remainder of a polynomial $\m$ at the division by $\l$, and then to compute the unique expression $\m=\eta\l+\r$ with $\r\in V_\l$, as we will see in Section~\ref{pro}. Our treatment simplifies and sharpens Cohn's approach of atomic factorizations discussed in \cite{c2}, \cite{c4}.

\section{The tool: Leavitt algebras}
\label{leat}
By definition (cf. \cite{leav1}) the
\emph{module type} of a ring is either \emph{IBN}, an abbreviation for  \emph{invariant basis number}, that is, an isomorphism between modules generated freely by $m$ and $n$ elements, respectively, implies $m=n$; or the first integral pair $(m,n)\, (1\leq m\leq n-1)$ such that modules generated freely by $m$ and $n$ elements, respectively, are isomorphic. This definition implies immediately the following obvious but important property of rings without IBN.
\begin{proposition}\label{mfir} If free modules of rank $m$ and $n\geq m+1$, respectively, are isomorphic, then any finitely generated module can be generated by at most $m$ elements. In particular, a ring  of  module type $(1,n), \,\,  n\geq 2$ is a Bezout ring, i.e., every finitely generated one-sided ideal is principal.
\end{proposition}

For the existence of rings of  arbitrary module type $(m,n)$, Leavitt invented the algebra $\Lem$ by inverting the matrix $(x_{ij})$ over the free algebra $\Lambda=K\langle x_{ij}\rangle \, (i=1, \cdots, m; j=1, \cdots, n)$ \cite{leav1}.  Although $\Le$ is now somewhat well-understood, $\Lem\, (2\leq m\leq n-1)$ 
is still a big puzzle.

\medskip

We present below an elementary and transparent approach to $\Le$, shortly $\Leo$, by listing results in such an order that their proof is immediate. For the details we refer to \cite{as}. 

By definition $\Leo$ is the localization of $\Lambda=K\langle x_1, \cdots, x_n\rangle=\Lambda_n$ inverting the row $(x_1, \cdots, x_n)$, i.e., $\Leo$ is generated  by elements  $x_i$ and  $x^*_i\, (i=1, \cdots, n)$ subject to relations 
\begin{align}\label{de3}
\su x_ix^*_i=1\qquad \& \qquad   x^*_jx_i=\begin{cases} 1& \text{if}\qquad j=i,\\ 0& \text{if}\qquad j\neq i\end{cases}.
\end{align} 
The images of $x_i, x^*_i$ in $\Leo$ are also denoted by $x_i, x^*_i$ themselves by convenience and by the fact that free algebras $\Lambda$ and $\Lambda^*$ are isomorphic canonically to the $K$-subalgebras $\bar \Lambda$ and ${\bar \Lambda}^*$ of $\Leo$ generated by $\{x_1, \cdots, x_n\}$ and $\{x^*_1, \cdots x^*_n\}$, respectively. This result is immediate by Theorem \ref{quo1} from which almost all basic properties of $\Leo$ can be read off easily. By putting
$b^*=x^*_{i_l}\cdots x^*_{i_1}$ for $b=x_{i_1}\cdots x_{i_l}\in \bB$, every element $r\in \Leo$ can be expressed (not necessarily uniquely) as a linear combination $\sum k_ib_ic^*_i \,(0\neq k_i \in K; b_i, c_i\in \bB)$. To verify Theorem \ref{quo1} we need first

\begin{lemma}\label{sub1} For any $r\in \Leo$, $r\in ({\bar \Lambda} \cap r\Leo)\Leo$ and ${\bar \Lambda}\cap r\Leo$ is a finitely generated right ideal of $\bar \Lambda$.  In particular,  if  $\frak R$ is a right ideal of $\Leo$ and   $R=\bar\Lambda \cap \frak R$, then 
 $\frak R=R\Leo$. Moreover $R$ is a finitely generated  right ideal of $\bar\Lambda$ if and only if $\frak R$ is principal. By symmetry, $\frak L=\Leo(\bar\Lambda^*\cap \frak L)=\Leo L$ holds for every left ideal $\frak L$  of $\Leo$ and $L$ is a finitely generated left ideal of ${\bar\Lambda^*}$ if and only if ${\frak L}$ is principal.\end{lemma}
\begin{proof} The equality in $\Leo$
$$1=\sum\limits_{i=1}^nx_ix^*_i=\sum\limits_{i=1}^nx_i\left(\sum\limits_{j=1}^nx_jx^*_j\right)x^*_i=\cdots=\sum\limits_{b, c\in \bB}^{|b|=|c|=l}bc^*\quad (l\in \N)$$
and $r\cdot 1=r=1\cdot r$ imply
\begin{align}\label{eq10}
r=\sum\limits_{b, c\in \bB}^{|b|=|c|=l}(rb)c^*=\sum\limits_{b, c\in \bB}^{|b|=|c|=l}b(c^*r) \quad (l\in \N)
\end{align}
with $rb\in \bar \Lambda$ and $c^*r\in {\bar \Lambda}^*$ for almost all $l$.  Consequently, if $r$ is represented as a linear combination
$r=\sum\limits_{i=1}^lb_ic^*_i$ with the least number $l$ of non-zero coefficients,  setting  $d=d_r=\max \{|c_i| \,|\, i=1, \cdots, l\}$, then for $m=d+1$ one has $rb\in r\Leo \cap \Lambda$ for all monomials $b$ of degree $m$. Thus  $r\in (r\Leo \cap \bar \Lambda)\Leo$ holds by \eqref{eq10} and hence  $\frak R=R\Leo$. Moreover, $r\Leo \cap \bar \Lambda$ is trivially generated by the finite set $\{rb\, |\, \deg b=m\}$, hence the proof is complete by Proposition \ref{mfir}.  
\end{proof}

\begin{remark}\label{def} {Recall that $I$ resp., $I^*$ is the two-sided ideal of $\Lambda$ resp., $\Lambda^*$ defined by the set $\{x_1, \cdots, x_n\}$ resp., $\{x^*_1, \cdots, x^*_n\}$ of variables. The corresponding ideal topology ${\frak T}\, (resp., {\frak T}^*)$ is a right (resp., left) Gabriel topology of $\Lambda$ (resp., $\Lambda^*$).} By \cite[Lemma~IX.1.6]{s1},  $Q_I(\Lambda)=\inlim_{l\in \N}\Hom_{\Lambda}(I^l, \Lambda_{\Lambda})$
 is a ring with respect to the multiplication induced by partial composition of maps. Another equivalent way is to identify $Q_I(\Lambda)$ with the subring of the maximal ring of right quotients of $\Lambda$ generated by $\Lambda$ and $\Hom_{\Lambda}(I_{\Lambda},\Lambda)$.
{By definition $Q_I(\Lambda)$ is the ring of right quotients of $\Lambda$ by the Gabriel topology $\frak T$.}
\end{remark}

\begin{theorem}\label{quo1} $\Leo$ is naturally isomorphic to the ring $Q_I(\Lambda)$ of right quotients of $\Lambda$ by the ideal topology ${\frak T}=\{I^l\, |\, l\in \N\}$ hence it is a flat left module over $\Lambda$. By symmetry, $\Leo$ is also isomorphic to the ring $Q_{I^*}(\Lambda^*)$ of left quotients of $\Lambda^*$ by the ideal topology ${\frak T}^*=\{(I^*)^l \,|\, l \in\N\}$ hence it is a flat right module over $\Lambda^*$. 
\end{theorem} 

\begin{proof} Since the $I^l_{\Lambda}$'s are freely generated by $n^l$ monomials of degree $l$, $\frak T$ is a perfect Gabriel topology by \cite[Propositions VI.6.10, XI.3.3 and XI.3.4(d)]{s1}. Hence
 $Q_I(\Lambda)={\inlim}_{l\in \N} \Hom_{\Lambda}(I^l, \Lambda_{\Lambda})$
 is a $K$-algebra, called the \emph{ring of right quotients} of $\Lambda$ by $\frak T$, as we already mentioned. {For every monomial $b \in {\bB} \, (|b|=l\geq 1)$ of positive degree $l$ let
$b^{\star}$ be the right $\Lambda$-module homomorphism
$b^{\star}: I^l\rightarrow \Lambda_{\Lambda}$ such that $$c\in I^l \cap {\bB} \mapsto b^{\star}(c)=\begin{cases} 1& \text{if}\qquad c=b,\\ 0& \text{if}\qquad c\neq b \end{cases}.$$}
In view of the equalities
$$I^l=\bigoplus\limits_{b\in \bB}^{|b|=l} b\Lambda\qquad \& \qquad \Hom_{\Lambda}(I^l_{\Lambda}, \Lambda)=\bigoplus\limits_{b\in \bB}^{|b|=l} {_\Lambda\Lambda}b^{\star}$$
and the fact that all $\Hom_{\Lambda}(I^l, \Lambda_{\Lambda})\, (l \in \N)$ are free left $\Lambda$-modules of rank $n^l$ one obtains that ${_{\Lambda}Q_I}(\Lambda)$ is flat. For $b=b_{i_1}\cdots b_{i_l}$ it is easy to check the equality $b^{\star}=b^{\star}_{i_l}\cdots b^{\star}_{i_1}$ in $Q_I(\Lambda)$. Hence $\Lambda$ and the $\, x^{\star}_i$'s generate $Q_I(\Lambda)$ and the row $(x_1,\cdots, x_n)$ is the inverse of the column $^t(x^{\star}_1, \cdots, x^{\star}_n)$. Therefore, from the definition of $\Leo$, there is a canonical map $\f:\Leo\rightarrow Q_I(\Lambda)$ induced by the identity map on $\Lambda$ and the rule $x^*_i\mapsto x^{\star}_i$ for  $i=1, \cdots, n$. Consequently, $\f$ is surjective and restricts to the identity on $\Lambda$ whence $\Lambda$ is a subalgebra of $\Leo$. 

For the injectivity of $\f$, assume $\f(r)=0$ for some $r\in \Leo$. Then, by Lemma~\ref{sub1}, $r\neq 0$ implies $0\neq r\m\in \Lambda$ for some monomial $\m$ and so $0\neq r\m=\f(r\m)=\f(r)\f(\m)=0$, a contradiction. Consequently, $\f$ is an isomorphism whence $\Leo$ is isomorphic to $Q_I(\Lambda)$.
\end{proof}
Interchanging the role of $x_i$ and $x^*_i$ one obtains
\begin{corollary}\label{sub2} The free algebra $\Lambda^*=K\langle x^*_1, \cdots, x^*_n\rangle$ is naturally isomorphic to the subalgebra ${\bar \Lambda}^*$  of $\Leo$. The rule $b\in {\bB}\mapsto b^*$ induces an involution of $\Leo$ which restricts to the canonical anti-isomorphism from $\Lambda$ to $\Lambda^*$.
\end{corollary}
Hence $\Lambda$ and $\Lambda^*$ are viewed naturally as $K$-subalgebras of $\Leo$. Moreover, $\Lambda^*$ and $\Lambda$ are isomorphic free algebras, for example, under the star-forgetting rule $b^*\mapsto {^{\tau}b} \, (b\in \bB)$. 

Since a {nonunital} subalgebra $A\subseteq \Lambda$ of codimension 1 is both a two-sided ideal and hence a free nonunital algebra of rank $n$, {as it is straightforward to verify}, Theorem \ref{quo1} provides a variable-free definition of $\Leo$ and sheds a light on "canonical" involutions of $\Leo$ and a close connection to the automorphism groups of free algebras.
\begin{corollary}\label{varifree} If $A\subseteq \Lambda$ is a {nonunital} subalgebra of codimension 1, then {$A$ is a two-sided ideal of $\Lambda$ and} $\Leo\cong Q_A(\Lambda)$, the ring of right quotients with respect to the ideal topology $\{A^l \, | \, l\in \N\}$.
\end{corollary}
As a ring of partial maps, every element $r\in \Leo$ is represented by a $\Lambda$-homomorphism $r: S_r\rightarrow R_r$ from the right ideal $S_r=\{\l\in \Lambda \, |\, r\l \in \Lambda\}$, the \emph{domain}, onto $R_r=\Leo r\cap \Lambda$, the \emph{range} of $r$.  Hence $S_r$ has finite codimension because it contains some power of $I$ and $R_r$ is finitely generated. However $R_r$ has not, in general, finite codimension. 
As a first obvious consequence, by comparing degrees of $\l(b)=\l b$ and $r(b)=rb$ for $\l \in \Lambda,\, r\in \Lambda^*$ and  $b\in \bB$ a monomial contained in the range $R_r$ of $r$, we have clearly the following well-known special case of \cite[Corollary 1.5.7]{aas}

\begin{corollary}\label{disj}   $\Lambda\cap \Lambda^*=K$ as subalgebras of $\Leo$.
\end{corollary}

\begin{proposition}\label{simple1}  If $n\geq 2$, then $\Leo$ is a purely infinite simple $K$-algebra, i.e., for every non-zero $r\in \Leo$ one has $srt=1$ for suitable elements $s, t\in \Leo$  
\end{proposition}
\begin{proof} 
Assume $n\geq 2$. For the simplicity of $\Leo$, let $0\neq r\in \Leo$  and  consider the elements  $0\neq \l=\sum\limits_{i=1}^l k_ib_i \in r\Leo \cap \Lambda$ with smallest $l$.  Among them, choose an element $\bar\lambda$ with  smallest degree. Then first by multiplying  an appropriate $x_i$ on the right if $\bar\l$ has non-zero constant and then on the left  by another $x^*_j\neq x^*_i$,  which is possible since $n\geq 2$, one can see that all $b_i$ must start with a same variable, say $x$, and $\l$ is without constant if it is not a non-zero scalar.  Hence $x^*\l$ has lower degree and so $\l$ must be a nonzero constant.  This implies  the purely infinite simplicity of $\Leo$.
\end{proof}

In the following corollary we list some basic properties of $\Leo$ that can be deduced from Theorem \ref{quo1}. In particular Claim (3) describes the open left ideals of $\Leo$.
\begin{corollary}\label{flatpro}
 Let $\frak L$ a left ideal  of $\Leo$ and  $L=\Lambda^*\cap \frak L$. Then:

   \begin{enumerate}
\item  $\frak L=\Leo L$ holds for every left ideal $\frak L$ of $\Leo$, so  $L$ is a finitely generated left ideal of $\Lambda^*$ if and only if  $\frak L$ is finitely generated.
\item The algebra  $\Leo$ is hereditary and projective-free. 
\item A left ideal $N$ of $\Lambda^*$ satisfies $\Leo N=\Leo$ iff $N$ contains a power of $I^*$.
\item $\frak L$ is finitely generated of finite colength if and only if $L$ is of finite codimension. 
\end{enumerate}
\end{corollary}
\begin{proof} Claim $(1)$ is given by Lemma~\ref{sub1}.

\noindent Claim (2) follows  from the fact that $\frak L=\Leo L=\Leo\otimes_{\Lambda^*} L$ by the flatness of $\Leo_{\Lambda^*}$ ensured in Theorem \ref{quo1}. Thus we conclude  since  free algebras,  {being free ideal rings,} are hereditary and projective-free.

\noindent  If $N$ contains $I^l$ for a suitable $l\geq 0$, then from the identity $1=\sum\limits_{b, c\in \bB}^{|b|=|c|=l}bc^*$ we get  $\Leo N=\Leo$. Conversely, if  $1=\sum\limits_{i=1}r_i\l^*_i$  for appropriate $r_i\in \Leo$ and  $\l^*_i\in N$, we can choose  an  integer $l$ big enough such that $b^*r_i\in \Lambda^*$ for any $b\in \bB$ with $|b|=l$ and any $r_i$. So we get  that  $I^l\leq N$. This shows   Claim (3).

\noindent  Since $\Leo$ is projective free, any principal left ideal is free and of infinite length. It follows easily that if $\frak L$  has finite colength, then it is essential. From the essentiality of $\frak L$ we get that also $L$ is essential in $\Lambda^*$. By the result  of Rosenmann and Rosset in  \cite[Theorem~3.3]{rr},  we conclude that  $L$ has finite codimension since it is finitely generated and essential.
To conclude Claim $(4)$ assume $L$ has finite codimension. In particular, again by \cite[Theorem~3.3]{rr},  it is finitely generated and hence  $\frak L$ is finitely generated. Notice that if $\frak L<\frak M<\Leo$, then $L<\frak M\cap \Lambda^*<\Lambda^*$, so $\frak L$ is of finite colength since  $L$ has finite colength.
\end{proof}

In the notation of the previous corollary, if $L$ is an arbitrary maximal left ideal of $\Lambda^*$ such that the associated primitive ideal $J=\ann_{\Lambda^*}\Lambda^*/L$ is not  contained in $I^*$, then $L$ cannot contain any power $(I^*)^l$. For, $(I^*)^l\subseteq L$ implies
$J+(I^*)^l\subseteq L$. However, $J+I^*=\Lambda^*$ implies $J=J\Lambda^*=J^2+JI^*$ and so $\Lambda^*=(J+I^*)^2=J+(I^*)^2=\cdots=J+(I^*)^l=J^2+(I^*)^l=\cdots=J^m+(I^*)^l$ for any $l, m\in \N$. Consequently, one has $L=\Lambda^*$, which is impossible. Hence $\Leo L\neq \Leo$ by Corollary \ref{flatpro}(3). Since $L$ is maximal, Corollary \ref{flatpro}(1) implies that $\Leo L$
is a maximal left ideal of $\Leo$.
Obvious examples for such maximal left ideals are ones which contain maximal two-sided ideals of the free algebra $\Lambda^*$ different from $I^*$. In this way one obtains a nice {large} class of {new simple modules over $\Leo$ with interesting endomorphism algebras. The subclass induced by maximal left ideals of $\Lambda^*$ which have finite codimension and do not coincide with $I^*$ is implicitly found in Ara's Proposition 3.3 and Theorem 5.1 \cite{ara1}}. More generally, the preceding argument shows that for pairwise different maximal two-sided ideals $J_i$ of $\Lambda^*$ together with positive integers $l_i \in \N$, we have $J^{l_1}_1+J^{l_2}_2\cdots J^{l_m}_m=\Lambda^*$. Hence if $L$ is a finite-codimensional left ideal and $J$ is the maximal two-sided ideal contained in $L$ such that inverse images of maximal two-sided ideals in the factor algebra $\Lambda^*/ J$ are different from $I^*$, then the extended principal left ideal $\Leo L$ satisfies $L=\Lambda^* \cap \Leo L$ and the modules $\Lambda^*/L$ and $\Leo/\Leo L$ have the same length. On the other hand, by Corollary \ref{flatpro}(4) every finitely presented left $\Leo$-module of finite length is isomorphic to $\Leo/\frak L$ such that $L=\Lambda^* \cap \frak L$ is finite-codimensional. 
Assume that $\frak L$ is a maximal left ideal of $\Leo$. Then the intersection $L=\frak L \cap \Lambda^*$ is not necessarily a maximal left ideal of $\Lambda^*$.  $L$ is  maximal only in the set of left ideals $M$ of $\Lambda^*$ satisfying $M=(\Leo M)\cap \Lambda^*$. 
Since $L=\Lambda^* \cap \frak L$ is  finite-codimensional there is
 $L=L_0\subseteq L_1\subseteq \cdots \subseteq L_l=\Lambda^*$ is a maximal ascending chain of left ideals with simple factors $L_{i+1}/L_i \, (i=0, \cdots, l-1)$. Moreover,  
$\Leo L_i=\Leo$ holds for all $i\geq 1$ and $L_1/L$ is a simple finite-dimensional $\Lambda^*$-module such that the associated primitive ideal is a finite-dimensional maximal two-sided ideal $J\neq I^*$. Consequently, if $M$ is a maximal left ideal of $\Lambda^*$ with $\Lambda^*/M\cong L_1/L$ one gets $\Leo/\Leo M\cong \Leo/\frak L$. Moreover, one obtains that all the other possible composition factors of $\Lambda^*/L$ are isomorphic to the trivial module $K=\Lambda^*/I^*$.
This establishes an one-to-one correspondence between isomorphism classes of finitely presented irreducible representations of the Leavitt algebra $\Leo$ and ones of finite-dimensional non-trivial (i.e., not isomorphic to the trivial $\Lambda^*$-module $K$) representations of $\Lambda^*$. {In this manner we obtain 
explicitly the above mentioned results in \cite{ara1}.}
\vskip 0.3cm
Localization theory shows that the torsion modules defined by the Gabriel topology ${\frak T}^*=\{(I^*)^l\,|\, l\in \N\}$ form a Serre {localizing} subcategory of $\Lambda^*$-Mod and $\Leo$-Mod is {the quotient category of $\Lambda^*$-Mod by this Serre category. Hence Ara's Proposition 3.3 and Theorem 5.1 \cite{ara1} are consequences of Theorem \ref{quo1} and the Gabriel's localization theory}. However, one need not to use sledge hammer to crack nuts and to both readers' benefit and the sake of the elementary self-containedness we present the special case of finitely presented modules of finite length.

Let $\frak {FD}$ be the full subcategory of finite-dimensional left $\Lambda^*$-modules. A finite-dimensional $\Lambda^*$-module is called a \emph{Leavitt} module if none of its simple subfactors is isomorphic to the trivial module $K$, that is, its annihilator ideal is a product of maximal two-sided finite-codimensional ideals different from $I^*$. The full subcategory of {$\Lambda^*$-Mod consisting of} all Leavitt modules is called the \emph{Leavitt category}. A finite-dimensional module is called a \emph{Leavitt-torsion} module if its simple subfactors are isomorphic to the trivial module $K$. Any $M\in \frak{FD}$ contains the largest submodule $T(M)$ which has all simple subfactors isomorphic to the trivial module $K$. Then $T(M/T(M))=0$ holds clearly. $T(M)$ is called the \emph{Leavitt torsion submodule} or simply \emph{torsion submodule} of $M$. If $T(M)=0$, then $M$ is a \emph{(Leavitt) torsionfree module}. {Leavitt modules are clearly torsionfree but torsionfree modules are not necessarily Leavitt modules since they may contain simple subfactors which are not contained in their socle and isomorphic to the simple trivial module $K$.} The full subcategory $\frak {LT}$ of Leavitt-torsion modules is clearly a Serre subcategory of $\frak {FD}$. If $M$ is a Leavitt module of length $l$, then the flatness of $\Leo_{\Lambda^*}$ implies that ${\frak M}=\Leo\otimes_{\Lambda^*}M$ is finitely presented and has the same length $l$. {Namely, if $0=M_0\subseteq M_1\subseteq\cdots\subseteq M_l=M$ is a composition chain of $M$, then the flatness of $\Leo_{\Lambda^*}$ implies that
$\Leo\otimes_{\Lambda^*}M_i$  $(i=0, 1, \cdots, l)$ form a corresponding composition chain of $\frak M$. Hence by comparing lengths and by the now routine use of the equality \eqref{eq10} together with similar reasoning for Corollary \ref{flatpro}, one can see that $M$ is isomorphic to
$\{1\otimes m\, |\, m\in M\}\subseteq \frak M$ as left modules over $\Lambda^*$. After this identification we have the  lattice isomorphism between the submodule lattices of $\frak M$ and $M$ induced by $N\subseteq M\mapsto {\frak N}=\Leo\otimes_{\Lambda^*}N=\Leo N\mapsto {\frak N}\cap M$, respectively.}  {Any module homomorphism $f\colon M\rightarrow N$ between Leavitt modules induces clearly the homomorphism $\Leo\otimes f:{\frak M}=\Leo\otimes_{\Lambda^*}M \rightarrow {\frak N}=\Leo\otimes_{\Lambda^*}N$. Furthermore, any homomorphism ${\frak f}:{\frak M}\rightarrow \frak N$ is induced by some homomorphism $\bar f\colon \bar M\rightarrow \bar N$ between submodules $\bar M, \bar N$ of $M, \, N$, respectively. The equality ${\frak f}=\Leo\otimes \bar f$ together with the above deduced lattice isomorphism of submodule lattices show $\bar M=M$ and $\bar N=N$.
Hence the Leavitt category of $\Lambda^*$ can be considered as a full and faithful subcategory of $\Leo$-Mod. 
On the other hand, if $\frak M$ is any finitely presented left $\Leo$-module of finite length, then $\frak M \cong \Leo/\frak L$ by Proposition \ref{mfir} and Corollary \ref{flatpro}(4) where $\frak L$ is a principal ideal of $\Leo$ of  finite colength. Since any homomorphism $f\colon \Leo/\frak L_1=\frak M_1\rightarrow \Leo/\frak L_2=\frak M_2$ between finitely presented modules of finite length is uniquely determined by the image $(1+\frak L_1)f\equiv r \mod \frak L_2 \, (r\in \Leo)$, the equality \eqref{eq10} shows that $f$ is uniquely determined by the images $(b^*+\frak L_1)f\equiv b^*r \mod \frak L_2$ with all $b^*r\in \Lambda^*$, where $b$ runs over all monomials of an appropriate degree $l$ big enough. Therefore $f$ is uniquely determined by a homomorphism ${\overline f}\colon\sum\limits_{|b|=l} \Lambda^*b^*/L_1={\overline M}_1\rightarrow \sum\limits_{|b|=l}\Lambda^* b^*r/L_2={\overline M}_2$ where $L_i=\Lambda^*
\cap \frak L_i$ and $M_i=\Lambda^*/L_i$ contains $\overline M_i \,(i=1, 2)$. Corollary \ref{flatpro}(3) implies that $M_1/{\overline M_1}$ is a torsion module. Note the fact that both $M_1$ and $M_2$ are torsion free.  Furthermore, any  homomorphism $g\colon M\subseteq M_1\rightarrow M_2$ such that $M_1/M$ is a torsion module, defines a homomorphism $1\otimes g\colon \frak M_1=\Leo\otimes M\rightarrow \frak M_2$ and $\overline {1\otimes g}$  coincides with $g$ on $M$. 

Take now an arbitrary $M\in \frak {FD}$. Then $M$ maps to a left $\Lambda^*$-submodule $\overline M=\{1\otimes m\,\,\,|\,\, m\in M\}$ of $\frak M=\Leo\otimes_{\Lambda^*}M$ whose kernel is exactly $T(M)$. If $f\colon \frak M \rightarrow \frak N=\Leo\otimes N\,\, (N\in \frak {FD})$ is any $\Leo$-homomorphism, then there is $l\in \N$ such that $b^*(1\otimes m)f\in {\overline N}=\{ 1\otimes n \,\,\, |\,\, n\in N\}\subseteq \frak N$ holds for all $b\in \bB$ with $|b|=l$. Hence $f$ is induced by a $\Lambda^*$-homomorphism from a submodule $\underline M$ of $M$ with $T(M/\underline M)=M/\underline M$ into a submodule $\underline N$ of $N/T(N)$. This implies Theorem 5.1 of Ara \cite{ara1} that the category $\frak {FPL}$ of finitely presented left $\Leo$-modules of finite length is the quotient category of $\frak {FD}$ by the Serre subcategory $\frak {TL}$. In particular, the Leavitt category is equivalent to a full subcategory of $\frak {FPL}$. {Notice that although the Leavitt category seems to be equivalent to the category $\frak {FPL}$ of finitely presented $\Leo$-modules of finite length, it is an open and important question whether they are actually equivalent.}
To conclude we have obtained a refinement of Theorem 5.1 of Ara \cite{ara1} as 
\begin{corollary}\label{arafp} Finitely presented $\Leo$-modules of finite length form the category $\frak {FPL}$ which is the quotient category of finite-dimensional $\Lambda^*$-modules by one of Leavitt-torsion modules, whence the {Leavitt category is fully embedded in $\frak {FPL}$ via the functor
$M\mapsto \Leo \otimes_{\Lambda^*} M=\frak M$ with $\End(\frak M)\cong \End(M)$.} Hence for any finite-dimensional factor algebra $A$ of $\Lambda^*$ such that $I^*$ maps to $A$, finite-dimensional left $A$-modules are fully embedded in $\frak {FDL}$.
\end{corollary}
\begin{remarks}\label{import1}
\begin{enumerate}
\item {Although finitely generated $\Leo$-modules are cyclic by Proposition \ref{mfir}, finitely generated Leavitt $\Lambda^*$-modules are not necessarily cyclic. {For example, if $J\neq I^*$ is a maximal finite-codimensional two-sided ideal of $\Lambda^*$ and $M=(\Lambda^*/J)^l$ with $l\geq 2$, then $M$ is clearly a finitely generated non-cyclic Leavitt $\Lambda^*$-module.} The submodule lattice of $M$ is isomorphic to that of the $\Leo$-module $\frak M=\Leo\otimes_{\Lambda^*}M$, which is a cyclic module. So $\frak M$ is isomorphic to $\Leo/\frak L$ for a suitable principal left ideal $\frak L$, by the fact that $M$ is finitely presented. However, $M$ and $\Lambda^*/L, \, (L=\Lambda^*\cap \frak L)$ are  not isomorphic $\Lambda^*$-modules.}
\item  There are several ways to realize a particular finite dimensional $K$-algebra as a factor algebra of $\Lambda^*$. For example, to every point $\underline{k}=(k_1, \cdots, k_n)$ of the affine space $K^n$ the factor algebra $\Lambda^*/I_{\underline{k}}, \,\, I_{\underline{k}}=\sum \Lambda (x^*_i-k_i)$ is isomorphic to $K$.
\item There are left ideals $L\subseteq \Lambda^*$ of finite codimension not containing any power of $I^*$ satisfying 
$L\neq \Leo L\cap \Lambda^*$ which are maximal left ideals. For example, $I^*\cap \sum\limits_{i=1}^n \Lambda^*(x^*_i-1)=L\neq \Leo L\cap \Lambda^*= \sum\limits_{i=1}^n \Lambda^*(x^*_i-1)$.\end{enumerate}
\end{remarks}

Corollary \ref{arafp} raises an important but vague question about {how  Leavitt modules can be constructed and classified. It is a big surprise that finite-dimensional $\Lambda$-modules $V_\l$ induced by noncommutative polynomials $\l \in \Lambda$ provide a simple transparent partial answer with remarkable consequences  via Theorem \ref{goal1}.}

We conclude this section with the following easy result which will be very useful in the sequel.

\begin{lemma}\label{pre2} If $\l \in \Lambda$ is a polynomial without constant, then  $\Leo\l=\sum \Leo{\gamma_i}$ for suitable polynomials $\gamma_i$'s with constant  and of degree lower than $\lambda$.
\end{lemma}
\begin{proof}
Assume $\lambda=\sum\limits_{i=1}^n x_i\lambda_{x_i}$, where the $\lambda_{x_i}$'s are the cofactors of $\lambda$ defined in the previous section. Then $\Leo\l=\sum\limits_{i=1}^n \Leo\l_{x_i}$ and we can repeat the argument until we get  $\Leo \l=\sum\limits_{j=1}^l \Leo \l_{b_j}$ for the  $b_j$'s   monomials of $\l$ of minimal degree. \end{proof}

\section{The proof}
\label{pro}
In this section we are going to  prove Theorem \ref{goal1}, by means of  the results on the Leavitt algebra $\Leo$ previously  presented. Recall in particular that both $\Lambda=K\langle x_1, \cdots, x_n\rangle$ and $\Lambda^*=K\langle x^*_1, \cdots, x^*_n\rangle$ are free subalgebras of $\Leo$. Without loss of generality $\l\in \Lambda$ is assumed comonic and  $|\l|\geq 1$, i.e., $\l=1+\sum x_i\l_{x_i}=1+{\underline \l}\neq 1$. Once $\lambda$ is fixed, we set   $L_\l=\Lambda^*\cap \Leo\l$ and $ {M_\l}=\Leo/\Leo\l$. 
\begin{lemma}\label{pre1} Let $\l \in \Lambda$ be a comonic polynomial. If $\a \in \Leo\l \cap \Lambda$, then $\a \in \Lambda\l$.  Hence $\Leo\l\cap \Lambda=\Lambda\l$.
\end{lemma}
\begin{proof} Put ${\underline \l}=-\delta$ for convenience. For $r\in \Leo$ with $\a=r\l$ one has $\a=r-r\delta$, so that  $r=\a + r\delta=\a+(\a+r\delta)\delta=\cdots=\a+\a \delta+\cdots +\a \delta^l+r\delta^{l+1}$ for all $l\in \N$. Since  there exists $l$ large enough such that  $r\delta^{l+1}\in \Lambda$, we get   $r\in \Lambda$ and so $\a\in \Lambda\l$.
\end{proof} 

\begin{corollary}\label{proper}
If $\lambda\in \Lambda$ is comonic and $|\l|\geq 1$, then $\Leo\l$ is a proper left ideal of $\Leo$ and $L_{\lambda}$ is a proper left ideal of $\Lambda^*$. In particular, if $\Leo \g=\Leo \g$ for nonzero polynomials $\d, \g \in \Lambda$ with constant, then $\g$ is a scalar multiple of $\d$.  
\end{corollary}
\begin{proof}
By the previous lemma,  polynomials in $\Lambda$ with constant of degree $\leq |\l|-1$ do not belong to $\Lambda\l$ and also not to $\Leo\l$. Hence $\Leo\l$ is a proper left ideal of $\Leo$. Moroever, $L_{\lambda}$ is  a proper left ideal of $\Lambda^*$ by Corollary \ref{flatpro}(1).
\end{proof}

\begin{lemma}\label{pre3} If $\l$ is comonic, then $\Lambda+\Leo\l=\Leo$, whence $\Lambda/\Lambda\l \cong \Leo/\Leo\l$ as $\Lambda$-modules.
\end{lemma}
\begin{proof} For ${\underline \l}=-\delta$, the congruence $1\equiv \delta \mod \Leo\l$ implies $1\equiv \delta^l \mod \Leo\l$ and hence $r\equiv r\delta^l \mod \Leo \l$  for any $l\in \N$ and any $r\in \Leo$. Consequently, choosing $l$ large enough such that $r\delta^l \in \Lambda$, we get $(r+\Leo\l)\cap \Lambda\neq \emptyset$ for all $r\in \Leo$. Hence $\Lambda+\Leo\l=\Leo$ holds and  $\Lambda/\Lambda\l \cong \Leo/\Leo\l$ as $\Lambda$-modules by the equality $\Leo\l \cap \Lambda=\Lambda\l$ of Lemma \ref{pre1} and the basic isomorphism theorem of Noether.
\end{proof}
\begin{remark}\label{star} Notice that $\Lambda/\Lambda\l$ becomes an $\Leo$-module by putting 
$$x^*_j(\g+\Lambda\l)=(x_j\star_{\lambda}\g)+\Lambda\l=-k\l_j+\g_j+\Lambda\l$$ for every $\g=k+\sum x_i\g_i\in \Lambda$, where $\g_j$ and $\l_j$ are the $x_j$-th cofactors of $\g$ and  $\l$, respectively. This definition is well-posed because $x_i\star_\l \l=0$ and $x_i\star_\l \Lambda \l \subseteq \Lambda \l$ hold for all $i=1, \cdots, n$, as it is routine to check.  In fact, the Leavitt algebra $\Leo$ with Lemma \ref{pre3} is a proper motivation for the  $\star_\l$-action defined in Section \ref{prif}. 

 Notice that  the $\star_{\lambda}$-action endows  $\Lambda\l$ and $\Lambda$  with a structure of left $\Lambda^*$-modules  {where multiplication by $x^*_i$ is exactly the $\star_\l$ action of $x_i$, that is,} 
$$x_i^*\times r=x_i\star_{\lambda} r$$
for any $r\in \Lambda$. However they are not modules over $\Leo$ because $\sum x_ix^*_i$ is not the identity on them.
\end{remark}

Facts on the  $\star_\l$-action listed in Section~\ref{prif} which are useful in  studying   injective envelopes of finite-dimensional $\Lambda$-modules are summarized as

\begin{corollary}\label{ess1} Let $\l=1+\sum x_i\lambda_{i}$. Then $\Lambda$ is a left $\Lambda$-module via the  $\star_\l$-action $x_j\star_{\lambda} \g=-k \l_j+\g_{j}$ for  $\g=k+\sum x_i\g_i\in \Lambda$. If $|\l|\geq 1$, then $V_\l\neq 0$ is essentially embedded in $\Lambda/\Lambda\l$.
\end{corollary}
The last claim of Corollary \ref{ess1} follows easily since $b\star_\l \g\in V_{\lambda}$ for any  monomial $b\in \bB$ of length $|\g|$.

\begin{lemma}\label{pre4} In $\Leo$ the equalities $V_\l+\Leo\l=V_\l \oplus \Leo\l=\Lambda^*+\Leo\l$ hold. Moreover $L_{\lambda}$ is a finite codimensional left ideal of $\Lambda^*$ and $V_\l \cong \Lambda^*/(\Leo\l\cap \Lambda^*)=\Lambda^*/L_\l$. In particular, $\codim L_{\lambda}= \dim V_\l$.
\end{lemma}

\begin{proof} 
Since $-1\equiv  {\underline \l} \mod \Leo\lambda$, for any $b^*\, (b\in \bB)$ we have  $-b^*\equiv b^*{\underline \l} \mod \Leo\l$. Hence  
any $b^*\, (b\in \bB)$ with $|b|\geq |\l|$ is either in $L_{\lambda}$ or congruent modulo $L_{\lambda}$ to linear combinations of $c^*\, (c\in \bB)$ with $|c|\leq |\l| -1$. Consequently, $\Lambda^*$ is the sum of $L_{\lambda}$ and the $K$-space generated by the finite set $\{b^*\,|\, |b|\leq |\l|-1, b\in \bB\}$ and hence $L_{\lambda}$ is a finite codimensional left ideal of $\Lambda^*$. In particular, if $|\l|=1$ then  $L_{\lambda}$ and hence $\Leo\l$ (by Corollary~\ref{arafp}) are maximal left ideals of $\Lambda^*$ and $\Leo$, respectively.

Assume now  that a cofactor $\l_b$ of $\l$ is not zero for some monomial $b\in \bB$, that is, $b$ is the head of some maximal monomial in $\l$. Then
\begin{align}\label{eq3}  
-b^*\equiv b^*{\underline \l}=\sum k_jt_b(j)^*+\l_b \Longrightarrow  -b^*-\sum k_jt_b(j)^*\equiv \l_b  \mod \Leo\l,
\end{align}
where $k_j$ is the coefficient of a proper head $h_b(j)\, (j\in \{1, \cdots, |b|-1\})$ of $b$ which is also a monomial in $\l$. This 
implies $\l_b\in \Lambda^*+\Leo\l$.
  Since by Lemma \ref{pre1} we get $\l_b\notin\Leo\lambda$,  we conclude     that $V_\l$ is canonically  embedded into  $\Lambda^*/L_{\lambda}$.
On the other hand, the congruences $-x^*_i\equiv \l_{x_i}$ modulo $\Leo\l$ and an obvious induction based on congruences \eqref{eq3},   imply that any monomial $c^*$ of length $\leq |\l|$ is a linear combination of cofactors $\l_b$ of $\l$ modulo $\Leo\l$, so that $V_\l$ maps surjectively onto $\Lambda^*/L_{\lambda}$. So we have that the $K$-vector spaces $V_{\lambda}$ and $\Lambda^*/L_{\lambda}$ are isomorphic, and hence $\codim L_{\lambda}= \dim V_\l$.  
Finally, by  congruences \eqref{eq3}, we easily get that  $V_\l+\Leo{\lambda}=V_\l \oplus \Leo\lambda= \Lambda^*+ \Leo\l$. 
\end{proof}
Notice that the proof of Lemma \ref{pre4} implies the following important equality
$$\Lambda^*+\Lambda^*\l=V_\l+\Lambda^*\l=V_\l \oplus \Lambda^*\l.$$

Together with the algorithm described in the proof of  \cite[Lemma~ 3.2]{rr}, the proof of Lemma \ref{pre4} yields  an algorithm to write down a strong Schreier basis of $\Lambda^*/L_{\lambda}$ which generates a direct complement $V$ of $L_{\lambda}$ with respect to $\Lambda^*$. Then one can use the projection of $\Lambda^*$ onto $V$ along $L_{\lambda}$ to write down a basis of $L_{\lambda}$. This algorithm for getting bases of finitely-codimensional left ideals plays a crucial role in finding polynomials similar to one with nonzero  constant term, {for instance see Examples \ref{ex0} and \ref{ex1}}.  

As a direct consequence of the proof to Lemma \ref{pre3} we have
\begin{corollary}\label{iso1} The  ${\Lambda}$-module $V_\l$ and the $\Lambda^*$-module $ {\Lambda^*}/{L_{\lambda}}$ are isomorphic with respect to the canonical ring isomorphism $\Lambda^* \cong \Lambda$. 
Moreover, if $\l$ is linear, then $\codim{L_{\lambda}}=1$  and $\Leo/\Leo\l$ is a simple $\Leo$-module whose endomorphism ring is $K$.
\end{corollary}
\begin{proof}
It is straightforward to verify that the $K$-isomorphism between $V_{\lambda}$ and $\Lambda^*/L_{\lambda}$ given in Lemma \ref{pre4} is $\Lambda$-linear  with respect to the canonical ring isomorphism $\Lambda \cong \Lambda^*$. 
The last statement follows by Corollary~\ref{flatpro}.
\end{proof}
By the results in Section~\ref{leat}, we immediately get the following 
\begin{corollary}\label{proof3} The $\Leo$-module $M_\l=\Leo/\Leo\l$ is finitely presented and of finite length. 
\end{corollary}

\smallskip

It is time to better understand why and how interrelations between $\Lambda, \Lambda^*$ and the  modules $\Leo/\Leo\l=M_\l$,   $V_\l$, $\Lambda^*/L_{\lambda}$ are crucial for the aim of this work. First, as we already mentioned, the rule $x_i\mapsto x^*_i$ induces both the natural canonical algebra isomorphism and the canonical anti-isomorphism between $\Lambda$ and $\Lambda^*$.  The canonical isomorphism induces the module isomorphism between $V_{\Lambda}$ and ${\Lambda^*}/L_{\lambda}$ as we already mentioned in Remark \ref{nre1} in the case of one variable . Moreover, thanks to results in Section~\ref{leat}, the $\Lambda^*$-module  $\Lambda^*/L_{\lambda}$ and the $\Leo$-module $\Leo/\Leo\l$ share important features via localization, and also in terms of finiteness  or dimensions. Thus  investigating the $\Lambda$-module  $V_\l$ or the $\Leo$-module  $\Leo/\Leo\l$  leads to a deeper study of $\l$ and turns out to be a direct, more appropriate way to study the factorizations of $\l$ than what is proposed by Cohn in \cite{c2}, \cite{c4}. In fact, it provides a way to find polynomials similar to $\l$, to check irreducibility to $\l$ as well as its factorization. The canonical anti-isomorphism induces the symmetry between left and right modules over free algebras, so that results on factorization and divisibility  we will get on the right  can be stated also on the left.
\smallskip
{As a first application of these ideas we have an easily used criterion for divisibility.
\begin{theorem}\label{test} If $\l$ is a {comonic} polynomial  and $\g$ is an arbitrary polynomial in $\Lambda$, then $\l$ is a right factor of $\g$ if and only if $b\star_\l \g=0$ for all monomials $b\in \bB$ of degree $|\g|$. There is an algorithm to find the left quotient $\d$ of $\g$ satsisfying $\g=\d\l$. Symmetrically, $\l$ is a left factor of $\g$ if and only if $\g \star_\l b=0$ for all monomials $b\in \bB$ of degree $|\g|$ and one can find successively the right quotient $\d$ of $\g$ satisfying $\g=\l \d$.
\end{theorem}
\begin{proof} By Lemma \ref{pre1}, $\l$ is a right factor of $\g$ if and only if $\g \in \Leo \l$. In view of the equality \eqref{eq10}, $\g \in \Leo \l$ holds if and only if $b^*\g \in \Leo\l$ for all monomials $b\in \bB$ of  degree $l$, where $l$ is an arbitrary positive integer. However $b^*\g$ is clearly congruent to $^{\tau}b\star_\l \g$ modulo $\Leo\l$, as it follows by Lemma \ref{pre4} and the obvious induction. Therefore $\l$ is a right factor of $\g$ if and only if $b\star_\l \g \in \Leo\l$ for all monomials $b$ with $|b|=|\g|$. Since the degree of $b\star_\l \g$ is at most $|\l|-1$, Lemma \ref{pre1} implies $b\star_\l \g=0$ if $|b|=|\g|$. On the other hand,
the equality $\g=\d\l$ implies $x_i\star_\l \g=\g_i\l$ as we have already seen. This reduces the computation of the left quotient $\d$ of $\g$ of one of lower degree. Thus the proof is complete.
\end{proof}}  

The following theorem is the  key tool in proving Theorem \ref{goal1}  and it is, in fact, an algorithm to find a greatest common divisor of two comonic  polynomials in $\Lambda$.

\begin{theorem}\label{main1} Let  $\g$ and  $\d$  be comonic polynomials in $\Lambda$. Then  there is a comonic polynomial $\l \in \Lambda$ such that  $\Leo \d+\Leo \g =\Leo \l$.
\end{theorem}
\begin{proof} Without loss of generality one can assume that $\d$ and $\g$ are proper polynomials with $|\d| \leq |\g|=m$ and none of them is contained in the left ideal of $\Leo$ generated by the other. We use induction on $m$. The claim is obvious in view of Corollary \ref{iso1} if $m=1$. So assume $m>1$ and  put $\fL=\Leo \d+\Leo \g$. 
If $\fL=\Leo$, one takes $\l=1$. Hence one can assume $\fL\neq \Leo$. Since $\g$ does not belong to $\Leo\d$, by Theorem~\ref{test}  there exits a monomial $b_1\in \bB$ with $|b_1|=\deg \g$ such that $0\neq b_1\star_\d \g=\d_1 \in V_\d$ and $|\d_1|\leq |\d|-1$. Since $\d_1=b_1\star_\d \g \in b^*_1\d + \Leo\d$ one has $\Leo\d_1\subseteq \Leo \g+\Leo\d=\fL$, and so $\d_1\in \fL$.
Hence we conclude $\Leo\delta_1+\Leo\delta\leq \fL$.
 It is possible that $\d_1$ is a polynomial without constant. However, by Lemma \ref{pre2} one can replace $\d_1$ by a comonic polynomial in $\Leo\d_1 \cap \Lambda$. Thus one can assume  $\d_1$ comonic. If $\g \notin \Lambda \d_1$, then one obtains in the same manner $0\neq b_2\star_{\d_1} \g=\d_2 \in V_{\d_1}\cap \frak L$ comonic, $|\d_2|\leq |\d_1|-1\leq |\d|-2$ and  $\Leo\d_2+\Leo\d\leq  \frak L$. If $\g \notin \Leo\d_2$ we continue this process and after at most $|\g|-1$ steps we find $\m \in \Lambda\cap \frak L$ with constant such that $|\m|\leq |\d|-1$ with $\g \in \Lambda \m$, so that $\Leo \m+\Leo\delta=\fL$.   
If $|\d| \leq m-1$, then the theorem follows from the induction hypothesis. If $|\d|=m$, then changing the role of $\d$ with $\m$ and $\g$ with $\d$ we will obtain again the claim after repeating the above argument. Therefore the proof is complete.    
\end{proof}

Thanks to Theorem~\ref{main1}, one obtains immediately the first claim of Theorem \ref{goal1} describing comonic irreducible polynomials in terms of simple modules,  as for one variable.

\begin{corollary}\label{step1}
Let $\lambda$ be a comonic polynomial in $\Lambda$. 
\begin{enumerate}
\item   $\lambda$ is irreducible if and only if $\Leo/\Leo\l$ is a simple left $\Leo$-module.
\item $\lambda$ is irreducible if and only if $\Lambda^*/L_\l$ is a simple $\Lambda^*$-module.
\item $\lambda$ is irreducible if and only if $V_\l$  is a simple $\Lambda$-module.  
\end{enumerate}
In particular, if $\lambda$ is irreducible, then  $D=\End(M_\l)\cong \End(\Lambda^*/L_\l)\cong \End(V_\l)$ is a finite-dimensional division $K$-algebra, hence  $D\cong K$ if $K$ is algebraically closed.
\end{corollary}
\begin{proof}
$1)$ Assume $\lambda$ irreducible.  In order to show that $\Leo\lambda$ is maximal  we have to see that $\Leo\lambda+\Leo l=\Leo$ for any $l\in \Leo$.
By Lemma~\ref{pre3} we can reduce to show that $\Leo\lambda+\Leo\alpha=\Leo$ for any $\alpha\in \Lambda$. By Lemma~\ref{pre2} we can assume $\alpha$ comonic, and so the result follow from Theorem~\ref{main1}. 

Conversely,  if $\l$ is reducible, then $\l=\gamma_1\gamma_2$, with $\gamma_i$ comonic and of degrees strictly less than $\lambda$. Hence $\Leo\lambda\leq \Leo\gamma_2\leq \Leo$ and, by Corollary~\ref{proper}, the inclusions are proper since $\gamma_2\notin \Lambda\lambda$. So, if  $\Leo\lambda$ is a maximal left ideal, then $f$ has to be irreducible.

$2)$ Assume $\l$ irreducibile.  Let $J$ be any proper left ideal of $\Lambda^*$ strictly containing $L_\l$.
Then $\Leo J=\Leo$ holds because $\Leo\l$ is a maximal left ideal of $\Leo$ by the previous point. Consequently, Corollary \ref{flatpro}(3) shows that $J$ contains a power $(I^*)^l$ for some $l\in \N$. This means that ${\overline J}=J/(I^*)^l$ is a left ideal of $A=\Lambda^*/(I^*)^l$ which is a finite-dimensional local ring. However, since by assumption $\l$ is comonic, its constant term is $1$ and  hence for a maximal monomial $b$ of $\l$ we get $b^*\l= b^*+\cdots +k_b$ where $k_b$ is the non-zero coefficient of $b$ in $\l$. So $b^*\l$ is in $L_\l$ and its  image in $A$ is a unit of $A$. Thus   ${\overline J}=A$ holds. This implies $J=\Lambda^*$, proving that $L_\l$ is maximal in $\Lambda^*$.

Conversely, if $L_\l$ is maximal in $\Lambda^*$, then $\Leo\l$ is maximal by results in Section~\ref{leat} and Corollary \ref{proper}. 

$3)$ This follows  immediately since $V_\d$ and $\Lambda^*/L_\d$ are isomorphic  with respect to the $\ast$-omitting ring isomorphism

\end{proof}

We note that when $n=1$, then the field $D$ of Corollary \ref{step1} is precisely the field extension $K[x]/K[x]\l$ of $K$ by $\l$.

\begin{corollary}\label{acc}If a left ideal $\frak L$ of $\Leo$ contains some $\g \in \Lambda$, then $\frak L=\Leo\l$ for some comonic $\l \in \Lambda$. In particular,  it  has finite colength and $\End(\Leo /\frak L)$ is finite dimensional.
\end{corollary}
\begin{proof} By assumption $\frak L \cap \Lambda^*$ contains $L_\g=\Leo\g \cap \Lambda^*$. By Lemma~\ref{pre2} we can assume $\g$ comonic. By Lemma \ref{pre4}, applying the modular law we get $\frak L\cap (\Lambda^*+ \Leo\g)=\frak L \cap (V_\g+\Leo\g)=W+\Leo\g$ where $W=V_\g \cap \frak L$. Hence Corollary \ref{flatpro}(1) yields $\frak L=\Leo\g+\Leo W$. Since $W$ is a finite-dimensional subspace of $\Lambda$ the proof is immediately completed by successive applications of Theorem \ref{main1}.
\end{proof}
The previous result has a nice important consequence.
\begin{corollary}\label{distri} If $\l \in \Lambda$ is comonic, then there is a bijection between right factors of $\l$ (up to a scalar) and left ideals of $\Leo$ containing
$\l$. Hence the submodule lattice of $\Leo/\Leo\l$ is isomorphic to the distributive lattice ${\bf L}(\Lambda\l, \Lambda)$ of principal left ideals of $\Lambda$ containing $\l$.
\end{corollary}
\begin{proof}
If $\frak L$ is a left ideal of $\Leo$ containing $\Leo\l$, then by Corollary~\ref{acc} there exists $\delta\in \Lambda$ such that $\frak L=\Leo \delta$, so that $\delta$ is a right divisor of $\l$.
\end{proof}
Note that the distributivity of the lattice ${\bf L}(\Lambda\l, \Lambda)$ is a deep result of Cohn \cite[Theorem 4.3.3]{c2}.

\smallskip
The next corollary shows that irreducible polynomials without constant generate $\Leo$.

\begin{corollary}\label{pritri} If $\p \in \Lambda$ is irreducible  without constant, then $\Leo\p=\Leo$. In particular, $\Leo\l=\Leo$ if $\l$ is a product of irreducible polynomials without constant.
\end{corollary}
\begin{proof}
Let  $\p \in I$ be  irreducible. By Corollary \ref{acc}, $\Leo\pi=\Leo\lambda$  for  a comonic  polynomial $\lambda\notin I $. By Lemma~\ref{pre1} we get that $\p$ is a multiple of $\lambda$. Thus the irreducibility of $\pi$   implies $\lambda=1$.
\end{proof}

As a  next  step to the proof Theorem \ref{goal1}, we show the equivalence of the similarity over $\Leo$ and $\Lambda$.

\begin{corollary}\label{unique1} Let $\l\in \Lambda$ be comonic and $\l=\d_1\cdots \d_m$ be a  factorization of $\l$ into irreducible factors. Then:
\begin{enumerate}
\item The $\Leo$-module $M_\l=\Leo/\Leo\l$ has a composition series  with composition factors the simple modules $\Leo/\Leo \d_i$, for $ i=1, \cdots, m$. 
\item The composition factors of $\Lambda$-module $V_\l$ contains  the simple modules $V_{\d_i}$, for $i=1, \cdots, m$. 
\item Given a monic polynomial $\g\in \Lambda$, $\g$ and $\l$  are similar   over 
$\Lambda$ 
if and only if they are similar over $\Leo$ .
\end{enumerate}
\end{corollary}
\begin{remark}\label{nre2} Together with Corollary \ref{flatpro}(3), Claim (3) of Corollary \ref{unique1} says that $V_\l$ has  length $\geq m$ and all the other possible composition factors are isomorphic to the trivial module $K$. We will show that there are no more composition factors and so $|V_\l|=m$ in Corollary \ref{sim0}.
\end{remark}
\begin{proof} $(1)$ Since $\l$ is comonic,   all  the $\d_i$'s  have constant term and we can assume they are comonic.
By Lemma~\ref{pre1},  $\Leo$ admits a descending chain $\Leo \supset \Leo\d_l \supset \Leo \d_{l-1}\d_l \supset \cdots \supset \Leo\d_2\cdots \d_l \supset \Leo \l$ of pairwise different left ideals. For $i=1, \cdots, m$ let $\g_i=\d_i\cdots \d_l$, so $\g_1=\l, \, \g_l=\d_l$ and $\g_{l+1}=1$. For any $i=1, \cdots, m$ let $T_i=\Leo \g_{i+1}/\Leo\g_i$.  Then  $T_i$ is a  non-zero left $\Leo$-module and if we denote by $\a_{i+1}$ the image of $\g_{i+1}$ in $T_i$, then $T_i=\Leo a_{i+1}$ and $\d_{i}a_{i+1}=0$ holds. So  we get $T_i\cong \Leo/\Leo{\d_i}$, and   $\Leo/\Leo\delta_i$ is simple since  $\d_i$  is irreducible. 

$(2)$ 
Put $L_i=\Lambda^* \cap \Leo \g_i \,(i=1, \cdots, l)$.
By Corollary~\ref{iso1} $L_l$ is a maximal left ideal of $\Lambda^*$ and $V_{\d_l}\cong \Lambda^*/L_l$ is a simple composition factor of $V_\l$. From Claim $(1)$ and by Corollary~\ref{flatpro} together with Corollary~\ref{iso1} again, for each $i=1, \cdots, l-1$ the $\Lambda^*$-module $L_{i+1}/L_i$ has one composition factor isomorphic to $V_i$. Furthermore by the comments after Corollary~\ref{flatpro} all possible other composition factors of all $L_{i+1}/L_i \,(i=1, \cdots, l-1)$ are the trivial $\Lambda^*$-module $K$. Consequently, Claim $(2)$ is verified.


$(3)$ By Proposition 3.3.4 \cite{c2} the similarity of $\g$ and $\l$ over $\Lambda$ implies an equality
$\g \a=\b \l$ for some $\a, \b \in \Lambda$ such that $\g, \b$ have no common left factor while $\a, \l$ have no common right factor. By $\Lambda\alpha \cap \Lambda\lambda \neq 0$ one sees that $\Lambda\alpha+\Lambda\lambda=\Lambda\m$ for some $\m \in \Lambda$ whence $\m$ must be a unit because $\alpha$ and $\lambda$ have no common right factor. 
Consequently, $\Lambda\a+\Lambda\l=\Lambda$ and hence the endomorphism $1\mapsto \a$ of $_{\Leo}\Leo$ induces a surjective homomorphism $f:{M_\g}\rightarrow {M_\l}$ by the equality $(\Leo \g)\a=\Leo (\g \a)=\Leo (\b \l)\subseteq \Leo\l$. Since both $M_\g$ and $M_\l$ have finite length one obtains $\lg({M_\l})\leq \lg( {M_\g})$ whence $\lg({M_\l})=\lg( {M_\g})$ by symmetry. This shows that $f$ is an isomorphism and so $\g$ and $\l$ are similar over $\Leo$. On the other hand,  an isomorphism between
$\Leo$-modules $M_{\g}$ and $M_{\l}$ is clearly a $\Lambda$-isomorphism between $\Lambda/\Lambda\g$ and $\Lambda/\Lambda\l$ in view of Lemma \ref{pre3}, whence the similarity over $\Leo$ implies the one over $\Lambda$. This completes  the proof.
\end{proof}
 
 \begin{remark}\label{n=1}
It is important to see clearly the role of $\l$ and $V_\l$ even in the case $n=1$. As we have already seen in Remark \ref{nre1}, for a comonic polynomial $\l$ and  its  reciprocal  polynomial  $\g$, the module $V_\l\cong K[x]/K[x]\g$ is not in general isomorphic to $K[x]/K[x]\l$.  Hence passing from $\Lambda=K[x]$ to $\Leo=K[x, x^{-1}]$ we have the parallel appearance of both $\l$ and $\g$ but they are still not similar. To make them similar, or more generally to make all irreducible polynomials defining the same field extension of $K$ similar, which is reasonable, one needs a more general notion of similarity of linear transformations. By passing to the corresponding companion matrix, it is proved in \cite[Proposition 1.6 (2)]{aa} that they, i.e., the associated companion matrices, are indeed similar matrices.
\end{remark}

By Corollaries \ref{step1} and \ref{unique1},  the proof of Theorem \ref{goal1} is complete by the verification of
\begin{corollary}\label{sim0} Two comonic polynomials $\g$ and $ \l \in \Lambda$ are similar if and only if the $\Lambda$-modules $V_\g$ and $V_\l$ are isomorphic.
\end{corollary}
\begin{proof} As usual, let  $L_\g=\Lambda^*\cap \Leo\g$ and  $ L_\l=\Lambda^*\cap \Leo\l$. By Corollary~\ref{iso1}, we show that $\g, \l$ are similar if and only if $N_\g=\Lambda^*/L_\g, N_\l=\Lambda^*/L_\l$ are isomorphic $\Lambda^*$-modules. Since the sufficiency is immediate, we need only to show the necessity, that is, $N_\g$ and $ N_\l$ are isomorphic if $M_\g$ and $ M_\l$ are isomorphic $\Leo$-modules. In view of Corollary \ref{arafp} it is enough to see that the length of $N_\l$ is $l$ if $\l=\p_1\cdots \p_l$ is a product of $l$ irreducible polynomials. We use induction on $l$. The claim for $l=1$ is exactly Corollary \ref{step1}. Assume the claim for $l-1\geq 1$ and put $\l=\p_1\cdots \p_l=\p_1{\overline \l}$. Denoting by $L_{\overline \l}=\Lambda^*\cap \Leo{\overline \l}$,  by  induction  $N_{\overline \l}=\Lambda^*/L_{\overline \l}$  has  length $l-1$ . To complete the induction it is enough to see that $L_{\overline \l}/L_\l$ is a  simple $\Lambda^*$-module. First notice that   $L_{\overline \l}/L_\l$ is not the zero module, since $L_{\lambda}\neq L{\overline{\lambda}}$ by Corollary~\ref{flatpro}. Consider the right multiplication by $\overline{\l}$ in $\Leo$. It induces an isomorphism between $\Leo/\Leo\pi_1$ and $\Leo\overline{\l}/\Leo\l$. Thanks to Corollary \ref{arafp} and Corollary \ref{step1}, we get an isomorphism between the simple module $\Lambda^*/L_{\p_1}$ and $L_{\overline{\l}}/L_\l$, from which we deduce that there exists a $y\in L_{\overline{\l}}$ such that its annihilator is $L_{\p_1}$. Assume now $L_{\overline{\l}}/L_\l$ is not simple and let us  conclude obtaining  a contradiction. Indeed, for any ideal $L$ properly between $L_{\l}$ and $L_{\overline{\l}}$, $\Leo L=\Leo\overline{\l}$ by Corollary \ref{flatpro}. Since $L_{\overline{\l}}/L_\l$ is finite dimensional, this means that its annihilator ideal contains a power $(I^*)^m$ for a suitable $m>0$. Consequently, $(I^*)^my\in L_\l$, so that $(I^*)^m\leq L_{\p_1}$. Then, arguing as in Corollary~\ref{step1}(2), in the local ring $\Lambda^*/(I^*)^m$ the left ideal $L_{\p_1}/(I^*)^m$ contains a unit and hence $\Lambda^*=L_{\pi_1}$. This is contradiction since $\pi_1$ is comonic and irreducible, so we conclude the proof.
\end{proof}

We conclude this section emphasizing how our approach of considering polynomials in $\Lambda$ as elements of $\Leo$ and studying their associated modules could be an effective tool for the study of the arithmetic of polynomials in noncommutating
variables. First of all, given  any finitely many arbitrary polynomials $\l_1, \cdots, \l_l\in K\langle x_1, \cdots x_n\rangle$, we can define  their \emph{generalized greatest common divisor}, shortly \emph{greatest common divisor}, as the  polynomial with constant $\d \in K\langle x_1, \cdots, x_n\rangle$  which is a generator of the left ideal $\sum\limits_{i=1}^l \Leo\l_i$. As an immediate consequence of Corollary \ref{acc}, this greatest common divisor $\d$ exists and is unique up to a scalar, even in the case that the ordinary least common multiple of the $\l_i$'s  is $0$. It is important to note the following: while the greatest common divisor of a polynomial with constant is itself up to a scalar, they are different in the case of a polynomial without constant, since  the greatest common divisor has constant  by definition. {For example, the equality $\Leo(x_1+x^3_1+x_1x^2_2+x_2+x_2x^2_1+x^3_2)=\Leo(1+x^2_1+x^2_2)$ shows by the definition that the greatest common divisor of the polynomial $\l=x_1+x^3_1+x_1x^2_2+x_2+x_2x^2_1+x^3_2$ is the comonic polynomial $1+x^2_1+x^2_2$.} Theorem \ref{main1} offers also an algorithm to find $\l$. Because of its importance, we rephrase this fact together with Theorem \ref{main1} as

\begin{theorem}
Given $\l$ and $\g$ monic polynomials in $\Lambda$, then there exists a unique monic polynomial $\d$ in $\Lambda$ such that $\delta$ is a right divisor of $\l$ and $\g$, and any other common right divisor is a right divisor of $\d$.
\end{theorem}

{For example, let $\l=1+x^2_1+x^2_2$. Then $V_\l=K+Kx_1+Kx_2$ has dimension 3. It is easy to see that $V_\l$ is a simple $\Lambda$-module with $x_1\star_\l 1=-x_1$ and $x_2\star_\l 1=-x_2$. Hence $\l$ is irreducible. Let $\g=1+x_1+x^2_1$. Then $x^2_1\star_\l \g=(x^2_1\star_\l 1)+(x^2_1\star_\l x_1)+(x^2_1\star_\l x^2_1)=-1-x_1+1=-x_1\neq 0$ implies that the greates common divisor of $\l$ and $\g$ is 1. If $\a=1+x^2_1+x_2+x_2x^2_1+x^2_2+x^3_2$, then one can check routinely that $b\star_\l \a=0$ for all monomials $b$ in $x_1, x_2$ of degree 3. Hence $\l$ is the greatest common divisor of $\a, \l$. Actually, $\a=(1+x_2)\l$ holds.}

Moreover one can solve the so-called  "membership problem" in $\Lambda$ by a relatively simple and quick algorithm {as an easy consequence of Theorem \ref{test}. Namely, let $\g=1+\sum x_i\g_i=1+\underline{\g}$ be a comonic polynomial  in $\Lambda$ and $\l=k_\l+\sum x_i\l_i$ an arbitrary polynomial of degree $m$. Let $R_\l=\{\d_b=b\star_\g \l\neq 0\ \,\, | b\in \bB, \, |b|=|\l|\}\subseteq V_\g$. According to Theorem \ref{test} $\g$ is a right factor of $\l$ if and only if this set $R_\l$ is empty. Hence $R_\l$ behaves like a set of "generalized remainders of $\l$ at division on the left by $\g$. If $R_\l$ is not empty, then 
$$\Leo\g+\Leo\l=\Leo\g+\sum\limits_{\d_b \in R_\l}\d_b$$
and for $\d_\l=\sum\limits_{\delta_b\in R_\l} b\d_ b$ one can find algorithmically $\r_\l \in \Lambda$ such that $\l=\r_\l\g+\d_\l$.}

{Hence we have an expression $\l=\r_\l\g+\d_\l$ as a final result of dividing $\l$ by $\g$. We summarize these results in the following theorem.}

\begin{theorem}\label{divi} Let $\g\in \Lambda$ be a comonic   polynomial   and $\l \in \Lambda$ of degree $m$. Then $\l \in \Lambda\g$ if and only if $b\star_\g \l=0$ holds for every monomial $b\in \bB$ of degree $m$. Moreover, if $R_\l=\{\d_b=b\star_\g \l\neq 0 \, |\, |b|=m\}$, then there is a unique  polynomial $\r_\l\in \Lambda$ satisfying  $\l=\r_\l\g+\d_\l$ where $\d_\l=\sum\limits_{\delta_b\in R_\l} b\d_ b$. {By symmetry, one has an algorithm to determine the right quotient $_\l\r$ and the set of remainders ${_\l}R=\{_b\d=\l\star_\g b\neq 0 \, |\, |b|=m\}$ together with the expression $\l=\g({_\l\r})+_\l\d$ where $_\l\d=\sum\limits_{_b\delta\in {_\l}R}{_b\d}b$.}
\end{theorem}

\begin{remarks}\label{algo1}
\begin{enumerate}
\item{For the particular computation it is worth to note that it is not necessarily to compute all $b\star_\g \l$ with $|b|=m=|\l|$. It is enough to take such monomials $b$ of a minimal degree with $b\star_\g \l=0$ or $b\star_\g \l \in V_\g$, respectively. The constraint $|b|=m$ is required only to grant that $b\star_\g \l$ is an element of $V_\g$ when it is not zero.}
\item Although the "membership problem" is solved, the division algorithm is completely different from the case of polynomials in commutating variables  despite of the fact that it looks formally almost the same. First the "generalized remainder" $\d_\l$ by dividing $\l$ by $\g$ is, in general, not an element of $V_\g$. Secondly, the degree of $\d_\l$ can be arbitrary. In fact, we have no idea to describe the set $\{\d_\l=\sum b\d_b \, | \, \l \in \Lambda\}$. It is clear that every polynomial $\m \in \Lambda$ with $|\m|\leq |\g|$ can appear as a remainder at a division by $\g$ if $\m$ is {not} a scalar multiple of $\g$ in case of equal degree. For example, if $\g=1-x_1$, then $\m=1-x_2\notin \Lambda \g$ but its "strict remainder" in our sense is $x^*_1\m=x^*_1\equiv 1$ mod $\Leo\g$. Even for $m=1-x^l_1x_2$ we have $(x^*_1)^l\m=(x^*_1)^l-x_2\equiv 1-x_2\equiv 1$ mod $\Leo\g$. {Therefore for the particular use it is reasonable to determine only the set $R_\l$.}

\item To check the simplicity of  $\Leo/\Leo\p$ where $\p$ has constant, one needs compute $\Leo\p+\Leo \a$ where $\a$ runs over all nonzero elements of $V_\p$. This is a definitely harder, exhaustive job than the memership problem. According to Theorem \ref{main1} it can be verified in finitely many steps for each element $\a \in V_\p$. Hence the exercise can be carried out in finitely many steps when the base field $K$ is finite. If $K$ is infinite, then $V_\p$ has unfortunately infinitely many elements although it has a finite dimension. Consequently, finding right factors and irreducible factorizations of a polynomial is more difficult job.

\end{enumerate}
\end{remarks} 

Finally, by Theorem~\ref{goal1} we know that,  given $\lambda\in\Lambda$,  its irreducible decompositions are unique up to similarity. For some particular polynomials we can improve this result obtaining  the uniqueness of the irreducible factorization in the classical sense.

\begin{corollary}\label{unis} If $\l=\p_1\p_2 \cdots \p_l$ is a product of polynomials $\p_i$ similar to an irreduclible polynomial $\p$, then this factorization of $\l$ is unique up to a scalar, that is, if $\l=\d_1\cdots \d_l$ is another prime factorization of $\l$, then each $\d_j$ is a scalar multiple of $\p_j$ for $j=1, \cdots, l$.
\end{corollary}
\begin{proof} Since $V_\l$ has the distributive submodule lattice by Corollary \ref{distri}, 
one obtains immediately that $V_\l$ and $\Leo/\Leo\l$ are uniserial modules, that is, their submodules form a unique finite chain if all of their simple subfactors are isomorphic. We then conclude the proof recalling that if $\Leo\d=\Leo\gamma$ for $\g$ and $\d$ in $\Lambda$ with constant, then by Corollary~\ref{proper} we get that 
$\delta$ is a scalar multiple of $\g$ 
\end{proof}

\section{Applications to factorizations of polynomials}
\label{fire}
 In the next two sections we give some hints for future directions that can be explored. 
As  shown by results and techniques exposed in the previous sections, free algebras and Leavitt algebras are deeply interrelated. Facts in free algebras are used to
investigate Leavitt algebras and vice versa. 
However, there is sometime the need to treat them separately in their own language,  to better understand their  applications and their relevance in the setting we choose to focus on. For instance,  from this point of view,  results of Section
\ref{pro} provide an algorithm to find factorizations of $\l \in \Lambda$ starting from the right. By symmetry or by using the canonical $^*$ involution of $\Leo$ and the anti-isomorphism between $\Lambda$ and $\Lambda^*$, they yield also a way to find  factorizations of $\l$ starting from the left.
Moreover, given a comonic polynomial $\l\in\Lambda$,  instead of
 $\Leo/\Leo\l$ and $\Lambda^*/L$ one can use $V_\l$ and the algorithm described in proof of Theorem \ref{main1} to decide irreducibility, similarity, factorization, and more generally, to do standard arithmetic of non-commutative polynomials. Since $\End_{\Lambda}(V_\l)$ is a finite-dimensional division $K$-algebra and $V_\l$ has a quite easily handled $K$-basis it is interesting to  closely study $\End_{\Lambda}(V_\l)$.  The first question arising  is  which non-algebraically closed field $K$ has a comonic irreducible polynomial $\l \in \Lambda$ such that  $\End_{\Lambda}(V_\l)$ is non-commutative.  The next one is to  determine  non-commutative finite-dimensional division $K$-algebras which are not representable as  $\End_{\Lambda}(V_\l)$ of irreducible polynomials $\l$, and so on.

\subsection{Companions to a polynomial: similarity, remainders etc.}\label{sip1}
Polynomials in not necessarily commuting variables form perhaps the most classical and quite hard subject of algebra. By definition,  a polynomial $\l$ is dependent only on variables appearing in its expression. Hence it  does not depend on the carrying free algebra $\Lambda$, which may have even infinitely many variables. From this point of view the finite-dimensional space $V_\l$ seems to be a good tool, a companion data structure as a simple combinatorial invariant. It plays the role of  the module of remainders and it is easy to describe. As we have already noticed,  it can be used to determine whether $\l$ is irreducible and then to find, at least theoretically, its factorizations  together with other related structures like its endomorphism ring and  similar polynomials. By  definition, to find a polynomial $\g$ similar to $\l$ one needs first to compute the factor module $M{_\l}=\Lambda/\Lambda\l\cong \Leo/\Leo\l$, then find another generator $m$ of $M_{\l}$  such that  $\ann_{\Leo}m=\{r\in \Leo \, |\, rm=0\in M_\l\} \neq \Leo\l$, and finally look for a polynomial $\g\in \Lambda$ such that $\ann_{\Leo}m=\Leo\g$. Hence $\g$ is similar to $\l$.

However, this direct approach with a few exception is not feasible. Fortunately, for a comonic polynomial $\l$ one can use  $V_\l$ to provide a way to compute all polynomials similar to $\l$, that is, to determine the similarity class $\frak{S}\lambda$ of $\lambda$. Since $V_\l$ is finite-dimensional, it is theoretically feasible to find all the generators of $V_\l$ and then to any generator $m\in V_\l$ the unique comonic polynomial $\g_m$  similar to $\l$ satisfying $\Leo \g_m=\Leo L_m$, where $L_m$ is the image of  $\ann_{\Lambda}m=\{\m\in \Lambda\,\, |\,\, \m\star m=0\}$ under the canonical isomorphism from $\Lambda$ onto $\Lambda^*$, provided that such a $\g_m$ exists.
These results are presented in the next  claim.
\begin{proposition}\label{simpo} Let $\l\in \Lambda$ be a comonic polynomial. Then,
up to non-zero scalars and automorphisms of $V_\l$, every comonic polynomial $\g$  similar to $\l$ determines uniquely a generator $m_\g$ of $V_\l$.
\end{proposition}
\begin{proof}  
 By Corollary \ref{sim0},  the $\Lambda$-modules $V_\g$ and $V_\l$ are isomorphic, whence $\g$ determines uniquely a generator $m_\g$ of $V_\l$ up to automorphisms of $V_\l$, where the automorphism group of $V_\l$ is isomorphic to that of the $\Leo$-module ${\Leo}/{\Leo\l}$. 
\end{proof}

\begin{remark}\label{simpo1} This assertion can be considered as a starting step describing similarity classes of polynomials. As the first trivial application one sees immediately that the similarity class of a linear polynomial with constant has exactly one element up to  scalars. In view of  \cite[1st ed. Proposition 3.3.4]{c2} similar polynomials provide a method to find different not permutated factorizations. 
In particular, until Proposition \ref{simpo1} there is no practical method finding similar polynomials. A thorough study of similar polynomials will be a subject of forthcoming works.   
\end{remark}
\begin{example}\label{ex0} {In commutative algebra it is well-known that  the maximal ideals of the polynomial algebra $A=K[x_1, \cdots, x_n]$ of codimension 1, are the ideals $\sum\limits_{i=1}^n A(x_i-k_i)$ where ${\frak k}=(k_1, \cdots, k_n)\in K^n$ runs over the points in the $n$-dimensional affine space $K^n$. For free associative algebra $\Lambda_n$, by passing to the corresponding factors, non-trivial one-dimensional $\Lambda_n$-modules are clearly parameterized by the polynomials $1-\sum\limits_{k_i\neq 0}k_ix_i$ for points ${\frak k}=(k_1, \cdots, k_n)\in K^n$. }
\end{example} 
\begin{example}\label{ex1} Let $\l=1+x_1x_2 \in \Lambda=\Lambda_2$. We compute the similarity class of $\l$. By definition the 2-dimensional $K$-space $V_\l=\{1, x_2\}$ is a $\Lambda_2$-module with respect to $x_1\star 1=-x_2$ and $x_2\star 1=0$. One can see that $\l$ is irreducible and $\{1+kx_2,\, x_2\,\, |\,\,  k\in K\}$ is the set of all generators of $V_\l$. Put $v=1\in V_\l$, then $\ann_{\Lambda}v$ is a $\Lambda_2$-module freely generated by $\{x_2, x^2_1, 1+x_2x_1\}$ whence $L_v$ is freely generated by $\{x^*_2, (x^*_1)^2,\, 1+x^*_2x^*_1\}$. Consequently, $\Leo L_v=\Leo\l$ because $\l=x_2x^*_2+x^2_1(x^*_1)^2+x_1x_2(1+x^*_2x^*_1)$. For $w=x_2\in V_\l$ we see that $\ann w$ is freely generated by $\{x_1, x^2_2, 1+x_1x_2\}$ and so by symmetry $\Leo L_w=\Leo(1+x_2x_1)$, that is, $\l$ is similar to $1+x_2x_1$. 

To describe the similarity class of $\l$, for any $0\neq k\in K$ let $v_k=1+kx_2$
be an arbitrary generator of $V_\l$ different from $1$ and $x_2$. We intend to look for the unique comonic polynomial $\g_k\in \Lambda$  such that  $\Leo\g_k=\Leo L_k$, where $L_k$ is the image of $\ann v_k$ in $\Lambda^*$ with respect to the canonical isomorphism from $\Lambda$ onto $\Lambda^*$. If such a $\g_k$ exists,  since $V_{\g_k}$ contains 1 and has $K$-dimension 2, $\g_k$ must have degree 2 in view of the definition of $V_{\g_k}$. Moreover, $V_{\g_k}$ has a $K$-basis $\{1, u=k_1x_1+k_2x_2\neq 0 \}$ whence at least one of the $k_1, k_2 \in K$ must be not zero. Furthermore $\g_k=1+x_1\g_1+x_2\g_2$ where $\g_1=l_1+l_2u, \g_2=m_1+m_2u$ with appropriate coefficients $l_1, l_2$ and $m_1, m_2$ from $K$ and at least one of the coefficients $l_2, m_2$ is not zero. Expressing
$$\g_k=\g=1+x_1[l_1+l_2(k_1x_1+k_2x_2)]+x_2[m_1+m_2(k_1x_1+k_2x_2)]$$
we have
$$\g_k=\g=1+l_1x_1+k_1l_2x^2_1+k_2l_2x_1x_2+m_1x_2+k_1m_2x_2x_1+k_2m_2x^2_2.$$
Hence $\g$ generates $\Leo L_k$ if $\Lambda^*\cap \Leo\g \subseteq L_k$, i.e.,  $\{x^*_ix^*_j\g \,\,|\, i, j=1, 2\} \subseteq L_k$. In  other words, we have to see whether
$$\{(x^*_1)^2+l_1x^*_1+l_1l_2, (x^*_2)^2+m_1x^*_2+m_2k_2, x^*_1x^*_2+m_1x^*_1+k_1m_2, x^*_2x^*_1+l_1x^*_2+k_2l_2\}$$
is contained in $L_k$. This means whether 
$$\{x^2_1+l_1x_1+l_1l_2, x^2_2+m_1x_2+m_2k_2, x_1x_2+m_1x_1+k_1m_2, x_2x_1+l_1x_2+k_2l_2\}$$
annihilates $v_k\in V_\l$.   
By $u\star v_k=(k_1x_1+k_2x_2)\star (1+kx_2)=-k_1x_2+kk_2$, the equality $0=(x^2_1+l_1x_1+l_2l_1)\star v_k=(-l_1+k)x_2+l_1l_2$ implies $l_1=k\neq 0$ and $l_2=0$. Furthermore $0=(x^2_2+m_1x_2+m_2k_2)\star v_k=m_1k+m_2k_2+kk_2m_2x_2$ implies $k_2m_2=0$ whence $m_1k=0$ and so
$m_1=0$ follows. Consequently, $0=(x_1x_2+m_1x_1+k_1m_2)\star v_k=k_1m_2+(kk_1m_2-k)x_2$ implies $k=0$ which contradicts  the assumption $k\neq 0$. This shows that $\Leo L_k$ cannot be generated by a polynomial from $\Lambda$ and so the similarity class of $\l$ has just two  comonic elements, namely, $\l$ and $1+x_2x_1$. 

This example shows two important facts. First not every principal
left ideal of $\Leo$ whose factor module is isomorphic to some $V_\l$ of a polynomial $\l \in \Lambda$ with constant (consider for instance $L_k$ as above), has a generator which is a polynomial with constant from $\Lambda$.  Consequently, not every finite-codimensional maximal left ideal of $\Lambda^*$ can be represented as $V_\l=\Lambda^* \cap \Leo\l$ for an appropriate  comonic polynomial $\l$ in $\Lambda$. The second fact is that it provides a way to determine the similarity class of polynomials with constant. Last but not least this example suggests a project to classify generators $v$ of $V_\l$ satisfying $\Leo L_v=\Leo\g$ for an appropriate $\g \in \Lambda$, where $L_v$ denotes the image of $\ann v$ in $\Lambda^*$ under the canonical isomorphism from $\Lambda$ to $\Lambda^*$. {In the same manner one can decide which 2-dimensional $\Lambda_2$-module can be described as $V_\l$ for an appropriate comonic polynomial of degree 2, and one can ask the same question for {small }dimensions.} 
\end{example}

\subsection{Factorizations of  polynomials}\label{fac}
Using Leavitt algebras to investigate polynomials requires that the studied  polynomials must be with constant. In view of Lemma \ref{pre2}, a polynomial $\l\in \Lambda$ without constant induces a finite set $\{\l_b\}$ of its cofactors,   where the  $\l_b$'s run over monomials of minimal degree in $\l$ and they are with constant. So the $\l_b$'s admit a greatest common divisor $\d$ by Theorem~\ref{main1}. Hence
$\d$ divides $\l$ if $\d$ is not a constant. In this way one can find all right factors of $\l$ with constant and by symmetry left factors with constant, too. However, this argument does not work in  case $\d\in K$. One can avoid this restriction by an appropriate substitution. For example, one can pass from  the polynomial $x^2_1x^3_2$ to the polynomial $(x_1+1)^2(x_2+1)^3$. {The feasibility of this trick depends on a problem whether there exists a polynomial $\l$ with zero constant which remains with zero  constant term  at any possible substitution of variables.}This opens the question whether the intersection of nonunital subalgebras of codimension 1 is zero as we have already pointed out. 

In any case, using Leavitt algebra may help to find factors with constant of a polynomial $\l$ without constant when $\Leo\l\neq \Leo$. For example, let 
$$\l=x_1x_2x_3x_2x_1+x_1x_2x_3+x_1x_2x_1+x_1+x_3x_2x_1 +x_3$$ 
be the polynomial considered by Vol\v ci\v c \cite{v1}.  Then by Corollary \ref{pre2} $\Leo \l=\Leo\l_1+\Leo\l_3=\Leo\l_3$ by $\l_1=\l_3+x_2x_3\l_3$, where $\l_1=x_2(x_3x_2x_1+x_1+x_3)+1$ and $\l_3=x_2x_1+1$. Consequently, $\l_3$ is a right factor of $\l=(x_1x_2x_3+x_1+x_3)\l_3$. On the other hand, one can use left factors of $\l$, that is, left factors of the image of $\l$ with respect to the $\ast$-adding isomorphism to obtain another factorization of $\l=(x_1x_2+1)(x_3x_2x_1+x_3+x_1)$. Iterating the argument of this example one obtains immediately together with Cohn's characterization of similarity \cite[Proposition 3.3.4]{c2}  the following result on factorizations of polynomials without constant.
\begin{proposition}\label{constant1} Let $\l \in \Lambda$ be a polynomial without constant and $\overline \l$ the image of $\l$ by the canonical isomorphism from $\Lambda$ onto $\Lambda^*$. If $\Leo \l\neq \Leo\neq {\overline \l}\Leo$, then $\l$ has two different factorizations $\l=\a \b=\g \d$  such that
$\a, \d$ have constant and $\b, \g$ have no constant.
\end{proposition}
\begin{proof} The assumption that $\Leo \l\neq \Leo$ implies that $\l$ has a factorization $\l=\g \d$ such that $\d$ has constant and $\g$ has no constant. The condition ${\overline \l}\Leo\neq \Leo$ implies by symmetry that ${\overline \l}={\overline a}{\overline \b}$, that is, $\l$ has a factorization $\l=\a \b$ where $\a$ has constant but $\b$ has no constant.
\end{proof} 
Therefore we have a way to produce polynomials having different factorizations. Just take a product  $\l=\a\b$ such that $\a$ is a product of irreducible polynomials with constant and $\b$ is a product of irreducible ones without constant. If $\Leo \l\neq \Leo$, then $\l$ admits another factorization $\l=\g \d$ where $\d$ has constant. In view of  the results
of Section \ref{pro},  it is a routine matter to check whether $\Leo\l=\Leo$ holds.  One can start with any polynomial $\l$ without constant and then use its left or right cofactors to check the inequalities $\Leo\l\neq \Leo\neq {\overline \l}\Leo$ where $\overline \l$ is the image of $\l$ by the canonical isomorphism in $\Lambda^*$. This argument can not be applied  for  polynomials $\l$ without constant satisfying $\Leo\l=\Leo=\l \Leo$ such that $\l$ remains without constant under any  variables substitution, even in the case when $\l$ has a factorization $\l=\a\b\g$ where $\a$ and $\g$ are polynomials without constant and $\b$ is one with constant. 

In order to find atomic factorizations of a comonic polynomial $\l\in \Lambda$, we focus first to a right comonic factor $\g\in \Lambda$ of $\l$. This means exactly that there is a positive integer $l\in \N$ such that $b\star_\g \l=0$ for all monomials $b\in \bB$ of degree $l$. So this is  a routine calculation to verify whether $\g$ is a right factor of another comonic polynomial. To compute the polynomial $\d$ satisfying $\l=\d \g$ we use the equality for each index $i\in \{1, 2, \cdots, n\}$
$$x_i\star_\g \l=x_i\star_\g \d\g=x_i\star_\g (\g+{\bar \d}\g)=\d_i\g.$$
Taking in account that $\d=1+{\bar \d}=1+\su x_i \d_i=1+\su x_i\d_{x_i}$ the determination of $\d$ reduces to compute
now a not necessarily comonic polynomial $\d_1$ with constant satisfying $\d_i \g=x_i\star_\g \l$ where for each $i\in \{1, 2, \cdots, n\}$ the polynomial $x_i\star_\g \l$ has lower degree. Thus in principle one can compute the left quotient $\d$ of $\l$ at the right division by $\g$.  

We are now in position to find atomic factorizations of $\l$. It is routine to compute $V_\l$. {Then one can continue as for integers. Take an arbitrary comonic polynomial $\g \in V_\l$ and verify if $\g$ does divide $\l$ by computing $b\star_\g \l$ for all monomials $b\in \bB$ with $|b|=|\l|$. In the positive case one finds further a factorization $\l=\d \g$, and the problem of factorization reduces to one of lower degree.  In principle this can be carry out without difficulty  in  case of a finite basic field $K$. However, there is a better way with possibly less computations as follows.}
By Corollary~\ref{unique1},  irreducible left factors of $\l$ correspond to minimal submodules of $V_\l$. So one can look for a minimal submodule $S$ of $V_\l$, then to a generator $m$ of $S$ and finally to an irreducible polynomial $\p$ determining $m$ as in \ref{sip1}. Next step is to check whether $\p$ divides $\l$ from the left. To check this matter, we need to use $W_\p$ which is the  simple finite-dimensional right $\Lambda$-module  associated to $\p$ defined in Remark \ref{leftright}, and verify whether $\l \star_\p b=0$ for all monomials $b \in \bB$ of degree $|\l|$. If $\p$ divides $\l$, then one can use the right symmetric version of the above described process to compute the right quotient $\l_1$ of $\l$ at the left division by $\p$, i.e., $\l=\p\l_1$ and continue in the same way for $\l_1$ instead of $\l$. If $\p$ does not divide $\l$, then 
look for the other  generators of $S$,  until we find a polynomial similar to $\p$ which divides $\l$. This method can be carried out step by step until an atomic factorization of $\l$ is found. This is a method to find atomic factorizations of $\l$ from the left hand side. To find right irreducible factors of $\l$, and then atomic factorizations of $\l$ starting from the right hand side, one need first to compute $W_\l$, generators of the simple submodules of $W_\l$ and then irreducible polynomials $\p$ representing them which become the right factors of $\l$. Of course the process can be very hard, complicated from a computational point of view. It is also worth to note that we have to use here both the left and right version of Theorem \ref{main1} to get all possible factorizations of $\l$.

\section{Final remarks and further open questions}
\label{infirank}
In this last section we continue a discussion of an individual polynomial in noncommuting variables. One companion structure of a polynomial is a primary decomposition. A polynomial $\l$ is called \emph{primary} if the associated module $V_\l$ is indecomposable. In contrast to the case of polynomials in commuting variables, primary polynomials in noncommuting variables have a great diversity. For example, $x_1x_2x_3$ and its permutated polynomials $x_ix_jx_k$ are all primary and none of them is similar to the other. Since $V_\l$ is finite-dimensional, it is a direct sum of its indecomposable submodules and this decomposition is unique up to isomorphism by Krull-Schmidt theorem. A challenging problem is to find factorizations of a comonic polynomial $\l$ into products of primary polynomials.

\subsection{Primary decomposition of a polynomial}\label{primary} For a comonic polynomial $\l \,(|\l|\geq 1)$  the $K$-vector space $V_\l$ is a  finite-dimensional $\Lambda$-module. Hence it decomposes uniquely as  a direct sum of indecomposable submodules $V_j$ up to  permutations. Indecomposable direct summands of $V_\l$ are called \emph{primary components} of $V_\l$. They are  determined by right factors $\p_j$ of $\l$ and can be considered as invariants of $\l$. However, we know almost nothing about them. For example we do not know whether $\l$ is their product. Since $V_\l$ has a distributive submodule lattice, Corollary~\ref{unis} shows  that, if $\l=\p_1\p_2 \cdots \p_l$ is a product of polynomials $\p_i$ similar to the same irreducible polynomial $\p$, then $V_\l$ is uniserial and hence it has a unique primary component.

It is an interesting project to study primary decompositions closely. For instance, in contrast to the commutative case, there are polynomials $\l$ such that $V_\l$ is indecomposable and their simple subfactors are pairwise non-isomorphic. For example $x_1x_2$ is a product of two non-similar irreducible factors  and $\Lambda/\Lambda x_1x_2$ is indecomposable. Substituting $x_1+1, x_2+1$ for $x_1. x_2$, respectively, one gets $\l=(1+x_1)(1+x_2)=1+x_1(1+x_2)+x_2$, a polynomial with constant. By definition $V_\l$ is a $K$-space generated by $1, 1+x_2$. It is now obvious by immediate calculation that $V_\l$ is indecomposable of length 2 with endomorphism ring isomorphic to $K$.

Finally,  one can  use the distributivity of the submodule lattice of $V_\l$ to show that not every finite-codimensional left ideal of $\Lambda^*$ can be represented as $L_\l=\Leo\l\cap \Lambda^*$ for some $\l \in \Lambda$, as we already  saw in Subsection \ref{sip1}. Namely, let $L_1, L_2$ be different finite-codimensional maximal left ideals of $\Lambda^*$ such that their associated primitive ideal is not $I^*$ and such that the factors $\Lambda^*/L_i \, (i=1, 2)$ are isomorphic. There are a great number of choices for such maximal left ideals, for example by $L_\l$ where $\l$ runs over similar irreducible polynomials with constant in $\Lambda$. Put $L=L_1\cap L_2$. Then $\Lambda^*/L=N$ and hence $\frak N=\Leo\otimes_{\Lambda^*}N$ is a direct sum of two  isomorphic simple modules. Therefore their submodule lattice is not distributive whence $L$ can not be represented as $L_\l$ for some $\l \in \Lambda$ with constant.

In fact, there are even maximal two-sided ideals of   $\Lambda^*$ with division factor algebras which cannot be represented as $V_\l$ for some comonic polynomial $\l \in \Lambda$. This claim is obvious for infinite-dimensional division factor algebras of $\Lambda^*$ because all $V_\l$ are finite-dimensional. Furthermore, if $K$ is the field $\mathbb{R}$ of the reals and $\mathbb{H}$ is the field of  the real quaternions given by the two-sided ideal $\langle{x_1^*}^2+1,  \ {x_2^*}^2+1, \  x^*_1x^*_2+x^*_2x^*_1 \rangle$, then $\mathbb{H}$ cannot be represented as $V_\l$ for any appropriate comonic polynomial $\l \in \mathbb{R} \langle x_1, x_2\rangle$. 
\subsection{Free algebra of arbitrary rank}\label{infi}
 It is clear that any element $\l$ of a free algebra contains only finitely many variables. Hence its study does not essentially depend on the other variables. Therefore from now on it is more
natural to assume that $\l$ is a polynomial in a free algebra $K\langle X\rangle$ on an arbitrary possibly infinite variable set $X$. In this case, $V_\l$ remains a finite-dimensional $\Lambda$-module but the equality $\Leo(X)\l=\Leo(X)(\Lambda^*\cap \Leo\l)$ does not hold. Here $\Leo(X)$ denote the Leavitt algebra by adding only a right inverse of the row matrix $X$, that is, it is generated by 
$\{x, x^*\, |\, x\in X\}$ subject to 
$$y^*x=\begin{cases}1& \text{if}\qquad x=y,\\ 0& \text{if}\qquad x\neq y\end{cases}$$
 for any two $x, y\in X$ when $X$ is infinite. For example, for infinite $X$ the right ideal $\Leo(X)(1-x)\cap \Lambda^*=\Lambda^*(1-x^*)+\sum\limits_{y\neq x}\Lambda^*y^*=L$ is still maximal but $\Leo(X)(1-x)\neq \Leo L$. However, $\Leo(X)\l$ is still a left ideal of colength $l$ if $\l=\p_1\cdots \p_l$ is a factorization of $\l$ into $l$ irreducible factors. To verify this result we need first an infinite version of Theorem \ref{main1}.
\begin{theorem}\label{main2} Let $\Lambda=K\langle X\rangle$ be the free algebra on an infinite set $X$ and $\Leo=\Leo(X)$ the corresponding Leavitt algebra. Then for any two polynomials $\d, \g \in \Lambda$ there is a comonic polynomial $\l \in \Lambda$  with $\Leo\d+\Leo\g=\Leo \l$. In particular, any submodule of the $\Leo$-module $_{\l}M=\Leo/\Leo\l\cong \Lambda/\Lambda\l$ is cyclic and can be generated by some comonic $\g\in \Lambda$. 
\end{theorem}
\begin{proof} From the definition one sees immediately that Lemma \ref{pre2} remains true whence by a trivial induction one can assume without loss of generality that both $\d, \l$ are comonic. Moreover, one can define a $\star_\l$-action of $\Lambda$ on itself for any  comonic polynomial $\l \in \Lambda$ in the same way as it was defined before Theorem \ref{goal1} whence $\Lambda/\Lambda\l$ becomes naturally a left $\Leo$-module isomorphic to $M_{\l}=\Leo/\Leo\l$. Therefore one has
$x\star_\l \a=-k_\a k^{-1}_\l \l_x+\a_x$ if $\l=k_\l+\sum\limits_{x\in X}x\l_x, \,\, \a=k_\a+\sum\limits_{x\in X}x\a_x, \, (0\neq k_\l, k_\a \in K)$ where $\b_x$ denotes the cofactor of $\b\in \Lambda$ with respect to $x\in X$. If we denote by $X_\a$ the subset of variables appearing in $\a$, then $X_\a$ is finite and $\sum\limits_{x\in X_\a}x^*$ acts as the identity on $\a$ mod $\Leo\l$, that is,
$$\sum\limits_{x\in X_\a}x(x\star_\l \a)\equiv \a \quad \mod \Lambda\l.$$
Consequently, if $b$ runs over the finite set of all monomials in $x\in X_\a$ of the degree $\a$, then at least one of these $b\star_\l \a$ is a non-zero element of $V_\l$, hence of degree $\leq |\l|-1$. These observations show that almost all results of Sections \ref{leat} and \ref{pro} together with their proof remain true if they are formulated by using $\Leo\l, V_{\l}$ and $M_{\l}(\cong \Lambda/\Lambda\l)$ and not involving $L_\l=\Lambda^*\cap \Leo\l, \, N_\l=\Lambda^*/L_\l$. Thus we have
\begin{corollary}\label{sup1} Let $\l \in \Lambda=K\langle X\rangle \,(|X|=\infty, \,\, |\l|\geq 1)$ comonic. Then $V_\l$ is a finite-dimensional
$\Lambda$-module with respect to the $\star_\l$-action which embeds essentially in $\Lambda/\Lambda\l \cong M_{\l}=\Leo/\Leo\l$. Both $V_\l$  and the $\Leo$-module $M_{\l}$ are simple with endomorphism ring isomorphic to $K$ if $\l$ is linear. For any $\a \in \Lambda$ not contained in $\Lambda\l$ there is a monomial $b$ in variables of $\a$ of degree $|\a|$ with $0\neq b\star_\l \a \in V_\l$. 
\end{corollary} 
In view of Corollary \ref{sup1} the proof of Theorem \ref{main1} works also in this case, and so Theorem \ref{main2} is verified.
\end{proof}
In view of Corollary \ref{sup1} and Theorem \ref{main2} the same argument in the last part of Section \ref{pro} implies the more general form of Theorem \ref{goal1} as follows.
\begin{theorem}\label{goal2} Let $\g, \l \in \Lambda=K\langle X\rangle \,$ be comonic polynomials  of positive degree where $X$ has an arbitrary rank, together with finite-dimensional left $\Lambda$-modules $V_\g, V_\l$  with respect to the $\star$-action defined by $\g, \l$, respectively. Let $\Leo(X)$ be the corresponding Leavitt algebra defined by
$\Lambda$.
\begin{enumerate}
\item $\l$ is an irreducible polynomial if and only if $M_{\l}=\Leo(X)/\Leo(X)\l$ and $V_\l$ are simple modules over $\Leo(X)$ and $\Lambda$, respectively. 
\item If $\l=\p_1\cdots\p_m$ is a factorization of $\l$ into a product of irreducible polynomials, then $m$ is the length of both $M_{\l}$ and ${_\Lambda}V_\l$ whence $m$ is an invariant of $\l$.
\item $V_\l\cong V_\g\Longleftrightarrow \Lambda/\Lambda\g \cong \Lambda/\Lambda\l\Longleftrightarrow M_{\g}\cong M_{\l}$, that is, $\g$ and $\l$ are similar over $\Lambda$ iff they are such over $\Leo(X)$.
\end{enumerate} 
\end{theorem} 

{{\bf Acknowledgement.} The authors express their gratitude to Professor Gene Abrams for his critical comments, questions and noble English assistance making this work transparent and enjoyable to the readers.}

\bibliographystyle{amsplain}

\begin{thebibliography}{10}
\bibitem{aa} G. Abrams, P. N. \'Anh, A matrix viewpoint for various algebraic extensions, \emph{Elemente der Mathematik}{\bf 75} (2020), 1-17, https://doi.org/1013-6018/20/020001-17
\bibitem{aas} G. Abrams, P. Ara, M. Siles Molina, \emph{Leavitt path algebras}, Lecture Notes in Mathematics {\bf 2191},  Springer 2018.
\bibitem{pna} P. N. \'Anh, Skew polynomial rings: the Schreier technique, \emph{Acta Math. Vietnamica}, https://doi.org/10.1007/s40206-021-00466-7
\bibitem{as} P.N. \'Anh, M.Siddoway, Leavitt path algebras are flat rings of quotients.  Arxiv version: arXiv:2108.11987
\bibitem{ara1} P. Ara, Finitely presented modules over Leavitt algebras, \emph{J. Pure Appl. Algebra}
{\bf 191}(1-2)(2004), 1 -- 21.
\bibitem{ab} P. Ara, M. Brustenga, Module theory over Leavitt path algebras and K-theory, \emph{ J. Pure Appl. Algebra} {\bf 214} (2010), 1131 -- 1151.

\bibitem{c2} P. M. Cohn, \emph{Free rings and their relations}, AP 1. ed. 1971, 2. ed. 1985.
\bibitem{c3} P. M. Cohn, \emph{An introduction to ring theory}, Springer UMS, 1999. 
\bibitem{c4} P. M. Cohn, \emph{Free ideal rings and localization in general rings}, Cambridge University Press 2006.
\bibitem{c6} P.M. Cohn, Non commutative unique factorization domains,  \emph{Trans. Amer. Math. Soc}, {\bf 109} (1963), 313-331.
\bibitem{c5} P.M. Cohn, Rings with a weak algortithm, \emph{Trans. Amer. Math. Soc}, {\bf 109} (1963), 332-356.


\bibitem{le1} J. Lewin, Free modules over free algebras and free group algebras: The Schreier technique, \emph{Trans. Amer. Math. Soc.} {\bf 145}(1969), 455-465.

\bibitem{leav1} W. G. Leavitt, The module type of a ring, 
\emph{Trans. Amer. math. Soc.} {\bf 103}(1)(1962), 113 -- 130.

\bibitem{rr} A. Rosenmann, S. Rosset, Ideals of finite codimension in free algebras and the fc-localization,
\emph{Pacific J. Math.} {\bf 162}(2)(1994), 351 - 371.

\bibitem {s1} Bo Stenstr\"om, \emph{Rings of quotients}, Grundl math. Wiss. 217, Springer 1975.
\bibitem{v1} J. Vol\v ci\v c, Determinantal zeros and factorization of noncommutative polynomials, talk on the Graz's conference \emph{ Rings and Factorizations}, July 10 -- 14, 2023.
\end{thebibliography}

\end{document}